\documentclass[11pt,a4paper,aps,reqno]{article}
\usepackage{titlesec}
\usepackage{etex}
\usepackage{inputenc}
\usepackage[english]{babel}
\usepackage{lmodern}					
\usepackage{textcomp}				
\usepackage{verbatim}			
\usepackage[toc,page]{appendix} 		
\usepackage{cite}
\usepackage[low-sup]{subdepth}
\usepackage{pict2e}

\usepackage{amsmath,					
			amsthm,
			amsfonts,
			amssymb, 
			mathtools,xypic}
\usepackage{physics}
\usepackage{enumitem} 
\usepackage{todonotes}
\usepackage{stmaryrd}
\usepackage[mathcal]{euscript}
\usepackage{mathrsfs}
\usepackage{slashed}
\usepackage{ytableau}
\usepackage{youngtab}
\usepackage{bibentry}
\usepackage{musicography}
\usepackage{tensor}
\usepackage[new]{old-arrows}

\usepackage{fonttable}
\DeclareMathAlphabet{\mathpzc}{OT1}{pzc}{m}{it}

\DeclareFontFamily{U}{matha}{\hyphenchar\font45}
\DeclareFontShape{U}{matha}{m}{n}{
      <5> <6> <7> <8> <9> <10> gen * matha
      <10.95> matha10 <12> <14.4> <17.28> <20.74> <24.88> matha12
      }{}
\DeclareSymbolFont{matha}{U}{matha}{m}{n}

\DeclareMathSymbol{\pm}            {2}{matha}{"08}
\DeclareMathSymbol{\mp}            {2}{matha}{"09}
\DeclareMathSymbol{\varleftarrow}{3}{matha}{"D0}
\DeclareMathSymbol{\varrightarrow}{3}{matha}{"D1}
\DeclareMathSymbol{\vee}           {2}{matha}{"5F}
\DeclareMathSymbol{\wedge}         {2}{matha}{"5E}
\DeclareMathSymbol{\leq}         {3}{matha}{"A4}
\DeclareMathSymbol{\geq}         {3}{matha}{"A5}
\DeclareMathSymbol{\in}            {3}{matha}{"50}
\DeclareMathSymbol{\owns}          {3}{matha}{"51}

\usepackage{bbold}

\usepackage[all]{xy}
\usepackage{tikz-cd} 
\usetikzlibrary{tikzmark}
\usetikzlibrary{arrows}

\usepackage{graphicx}				
\usepackage{tabularx}				

\usepackage{authblk}   

\newtheorem{theorem}{Theorem}[section]
\newtheorem{corollary}[theorem]{Corollary}
\newtheorem{lemma}[theorem]{Lemma}
\newtheorem{proposition}[theorem]{Proposition}

\theoremstyle{definition}
\newtheorem{defi}[theorem]{Definition}
\newtheorem{remark}[theorem]{Remark}
\newtheorem{example}[theorem]{Example}

\usepackage{cfr-lm}

\usepackage[hmargin=2.2cm,vmargin=2.5cm]{geometry}		
\numberwithin{equation}{section}

\makeatletter
\DeclareRobustCommand{\Lcorner}{\mathbin{\mspace{1mu}\text{\L@corner}\mspace{1mu}}}
\newcommand{\L@corner}{%
  \setlength{\unitlength}{\fontcharht\font`T}%
  \begin{picture}(0.8,0)
  \roundcap
  \Line(0.1,0.95)(0.1,0.05)
  \Line(0.1,0.05)(0.7,0.05)
  \end{picture}%
}
\makeatother

\newcommand{\Loc}{{\scriptscriptstyle\mkern 1mu\Lcorner}}

\newcommand{\Hom}{\mathrm{Hom}}
\newcommand{\End}{\mathrm{End}}

\newcommand{\nn}{\nonumber}

\newcommand{\ad}{\mathrm{ad}\,}

\makeatletter
\newcommand{\thickhline}{%
    \noalign {\ifnum 0=`}\fi \hrule height 1pt
    \futurelet \reserved@a \@xhline
}
\newcolumntype{'}{@{\hskip\tabcolsep\vrule width 1pt\hskip\tabcolsep}}
\makeatother
\newcolumntype{"}{@{\hskip\tabcolsep\vrule width 1.5pt\hskip\tabcolsep}}
\makeatother

\newcommand{\scr}{}
\renewcommand{\mathscr}{\mathcal}
\newcommand{\scrG}{\scr{G}}

\newcommand{\scrU}{\scr{U}}

\usepackage[colorlinks]{hyperref}			

\newcommand{\fg}{{\mathfrak g}}

\titleformat{\section}
{\normalfont\bfseries\large\sffamily}{\textsf{{\thesection}}}{1em}{}
\titlelabel{\textsf{{\thetitle}}\quad}

\titleformat{\subsection}
{\normalfont\bfseries\sffamily}{\textsf{{\thesubsection}}}{1em}{}
\titlelabel{\textsf{{\thetitle}}\quad}

\titleformat{\subsubsection}
{\normalfont\bfseries\sffamily}{\textsf{{\thesubsubsection}}}{1em}{}
\titlelabel{\textsf{{\thetitle}}\quad}

\title{Graded Lie superalgebras from embedding tensors}

\author[a]{Sylvain Lavau\thanks{Corresponding author: lavau@math.univ-lyon1.fr}}
\author[b]{Jakob Palmkvist}
\affil[a]{\small \emph{Department of Mathematics, Galatasaray \"Universitesi,  \c{C}\i ra\u{g}an Caddesi No:36, 34349 \.{I}stanbul, Turkey}}
\affil[b]{\small \emph{Mathematical Sciences, Chalmers University of Technology, SE-412 96, G\"oteborg, Sweden}}
\date{\today}

\nobibliography*
\begin{document}

\maketitle

\abstract{We show how various constructions of $\mathbb{Z}$-graded Lie superalgebras
are related to each other.
These Lie superalgebras 
have a Lie algebra $\fg$ as the subalgebra
at degree $0$, an odd $\fg$-module $V$ as the subspace at degree $1$,
and an embedding tensor as an element at degree $-1$.
This is a linear map from $V$ to $\fg$
satisfying a quadratic constraint, which 
equips $V$ with the structure of a Leibniz algebra.}

\tableofcontents

\newpage

\section{Introduction}

Building up on advances on the reduction of eleven-dimensional supergravity 
\cite{CremmerSupergravityTheory1978} to four dimensions through the T-tensor 
\cite{dewit8Supergravity1982-local, dewit8Supergravity1982, dewittEmbeddingGauged81985}, the notion of an
\emph{embedding tensor} was introduced by physicists in order to account for gauging procedures in supergravity \cite{nicolaiMaximalGaugedSupergravity2001, dewitLagrangiansGaugingsMaximal2003, deWit:2005hv, dewitGaugedSupergravitiesTensor2008, dewitEndPformHierarchy2008}. Following this first account, 
the notion has evolved into a cornerstone of research in mathematical physics, offering a powerful approach to gauged supergravity theories, as well as an inspiration for developments in algebra \cite{kotovEmbeddingTensorLeibniz2020, shengControllingLinfinityAlgebra2021, rongNonabelianembedding2023, tengEmbeddingTensors2024, CaseiroEmbeddingTensors2024, teng3Leibniz2025}.

From a mathematical perspective, the {\it tensor hierarchy} that the embedding tensor gives rise to has been understood to carry a graded Lie superalgebra structure. Original accounts of this structure were provided by one the authors \cite{palmkvistTensorHierarchyAlgebra2014, greitzTensorHierarchySimplified2014}, while the other author contributed 
a functorial relationship between embedding tensors and differential graded Lie algebras \cite{lavauLieAlgebraCrossed2023, lavauCorrigendumLieAlgebra2023}. At the time, it was unclear under which conditions these respective graded structures coincide. This paper is dedicated to elucidate those conditions, and to provide a concise and self-contained introduction to the
algebraic key notions at work in the construction.

To be more precise, let $\mathfrak{g}$ be a real or complex Lie algebra, let $V$ be a module over $\mathfrak{g}$,
and let
$\Theta$ be a linear map $\Theta:V\to\mathfrak{g}$ satisfying the condition
\begin{equation}\label{eq:quadratic0}
\Theta(\Theta(u)\cdot v)=[\Theta(u),\Theta(v)]\,
\end{equation}
for any $u,v\in V$, where the dot denotes the action of $\fg$ on $V$. (In the applications to gauged supergravity, the linear map $\Theta$ is the {embedding tensor} which describes how the gauge algebra is embedded in the global symmetry algebra $\fg$.)
Associated to these data, a graded Lie superalgebra $\mathbb{T}$ was canonically constructed in \cite{lavauLieAlgebraCrossed2023, lavauCorrigendumLieAlgebra2023}.
In this graded Lie superalgebra, $\Theta$ is an element sitting at degree $-1$, the Lie algebra $\fg$ is the subalgebra
at degree $0$ and $V$ is the subspace at degree $1$.

Graded Lie superalgebras with the same interpretation of these elements had already been constructed in 
\cite{palmkvistTensorHierarchyAlgebra2014} and called 
\textit{tensor hierarchy algebras}. 
The initial data of this construction did not include a particular element $\Theta \in \mathrm{Hom}(V,\mathfrak{g})$
but a subspace of $\mathrm{Hom}(V,\mathrm{End}\,V)$. 
Obviously, the $\fg$-module structure on $V$ gives rise to a representation $\rho:\fg\to\mathrm{End} \, V$, 
so that the composite linear map $\rho \circ \Theta$ is an element of $\mathrm{Hom}(V,\mathrm{End}\,V)$,
hence enabling a connection between the two approaches.
More generally, given the vector space $V$ and 
a subset $T$ of
$\mathrm{Hom}(V,\mathrm{End}\,V)$ a graded Lie superalgebra $\scr P(V,T)$
can be constructed as in \cite{palmkvistTensorHierarchyAlgebra2014}.
In the original construction, only certain finite-dimensional Lie algebras $\fg$ were considered,
and
in addition to the $\fg$-module structure on $V$, 
the subset $T$ was in \cite{palmkvistTensorHierarchyAlgebra2014} 
also a $\fg$-module
satisfying a 
certain representation constraint, again motivated by applications to gauged supergravity
(more specifically, by supersymmetry in these theories).
There are other constructions of tensor hierarchy algebras, and generalisations
to cases where $\fg$ is infinite-dimensional, originating from the representation constraint in the finite-dimensional cases~\cite{Bossard:2017wxl,carboneGeneratorsRelationsLie2019,cederwallTensorHierarchyAlgebras2020a}.
However, in the present paper we neglect the representation constraint but
focus on the quadratic constraint (\ref{eq:quadratic0})
which has no role in these constructions, and therefore we will not say more about them here.

The main purpose of the present paper is to show how the two constructions above are related to each other.
In particular, we will show
that if $\fg$ is simple, $V$ is faithful and $\Theta\neq0$,
then $\scr P(V,T)$ is isomorphic to $\mathbb{T}$ for a certain $T$ modulo
ideals at degree $3$ and higher.
We also take the opportunity to give a rather detailed account of how our constructions of
graded Lie superalgebras are related
to other ones that have appeared in the literature (such as
prolongations and extensions of a structure called a \textit{local} Lie superalgebra), and how various properties
of these graded Lie superalgebras
are related to each other.

The paper is organised as follows. In Section~\ref{item12} we review the basics of graded Lie superalgebras, in particular
Kac's construction of maximal and a minimal ones from local Lie superalgebras.
We elaborate on these in Section~\ref{sec3} in order
to introduce the universal graded Lie superalgebra of a vector space, as defined by Kantor
\cite{Kantor}. We then explain how this algebra is related to the notion of prolongation, and we conclude this section on defining the graded Lie superalgebra $P(V,T)$ mentioned above. In Section~\ref{sec4} we provide a summary of the construction of the canonical algebra induced by a triple $(\fg,V,\Theta)$, upon factoring out 
ideals of $\mathbb{T}$ concentrated in degrees $3$ and higher. Denoting the resulting graded Lie superalgebra by ${L}=L(\fg,V,\Theta)$, we conclude the section by showing under which condition it is isomorphic to
$P=P(V,T)$ for a certain $T$. 
We conclude Section~\ref{sec4} with several important examples, while Section~\ref{sec5} gives some insights about the interpretation and outcomes of such constructions in potential mathematical applications.

\medskip
\noindent
\textit{\underline{Conventions:}} The ground field is $\mathbb{R}$ or $\mathbb{C}$, when not explicitly specified.
Throughout the paper, the $\mathbb{Z}$-grading 
is reversed compared to standard conventions in mathematics literature, in particular
\cite{lavauLieAlgebraCrossed2023, lavauCorrigendumLieAlgebra2023} (but agrees with
related papers where
tensor hierarchies in physics correspond to positive degrees of a tensor hierarchy algebra).
If $V$ is a $\mathbb{Z}$-graded vector space, the notation $V[\pm1]$  
denotes the vector space $V$ whose degree has been shifted by $\mp1$.
This notation originates from geometry where functions and differential forms are prominent, and is justified by the observation that the dual vector space of $V[1]$ is concentrated in degree $+1$. We take the convention that, to every element  $u\in V$ corresponds an element $u[-1]\in V[-1]$ and $u[1]\in V[1]$. (Above in the introduction, we have disregarded such shifts for the sake of readibility. In the main part of the paper, the corresponding expressions are given with the shifts taken into account.)

\medskip
\noindent
\textit{\underline{Acknowledgments:}} 
This work was initiated in January 2023
during the workshop \textit{Higher Structures, Gravity and Fields} at
the Mainz Institute for Theoretical Physics (MITP)
of the DFG Cluster of Excellence PRISMA* (Project ID 39083149). Significant progress
was made in the spring of 2025 during the program
\textit{Cohomological Aspects of Quantum Field Theory} at the Institut Mittag-Leffler
and in June 2025 during the \textit{5th Mini Symposium on Physics and Geometry} at the
RBI Zagreb. We are grateful to the organisers of all these activities for hospitality and support.
Part of the research was also conducted at the Aristotle University of Thessaloniki and at \"Orebro University,
the previous affiliations of SL and JP, respectively.
Finally, we would like to thank Thomas Basile, Martin Cederwall, Jens Fjelstad, Boris Kruglikov, Dimitry Leites, Henning Samtleben and Andrea Santi for insightful discussions
and helpful answers to questions.

\section{Graded Lie superalgebras} \label{item12}

Most of the definitions in this section
can be found for Lie algebras in \cite{kacSimpleIrreducibleGraded1968} (Section 1.2)
and were generalised to Lie superalgebras in \cite{kacLieSuperalgebras1977} (Section 1.2.2).

All vector spaces and algebras that we consider are graded with respect to both
$\mathbb{Z}$ and $\mathbb{Z}/2\mathbb{Z}$,  and the two gradings are compatible with each other in the sense
that odd elements have odd $\mathbb{Z}$-degree, and even elements have even $\mathbb{Z}$-degree. In the following, we will not recall every time that vector spaces are graded. In particular, 
graded Lie superalgebras with such gradings will be called just Lie superalgebras.

We thus assume a decomposition $U=\bigoplus_{k\in\mathbb{Z}} U_{k}$ for any vector space $U$.
We write $x=\sum_{k\in\mathbb{Z}}x_k$ for any $x\in U$, where $x_k\in U_k$, and denote the $\mathbb{Z}$-degree
by $|x_k|=k$. We also use the notation
$U_{\pm}=\bigoplus_{k=1}^\infty U_{\pm k}$ for the positive and negative parts of $U$,
and $U_{p\pm}=\bigoplus_{k=0}^\infty U_{p\pm k}$ for any $p\in\mathbb{Z}$.
Morphisms preserve gradings, and a subspace of a vector space inherits the grading. In particular, when
we talk about ideals of Lie superalgebras we always mean ideals. 
The tensor product ${U}\otimes{V}$ of two vector spaces ${U}$ and ${V}$
is a vector space with the $\mathbb{Z}$-grading given by
$({U}\otimes{V})_k=\bigoplus_{i+j=k}{U}_i \otimes {V}_j$. 
A morphism of vector spaces $\varphi:U\to V$ can thus be considered as a
family $\varphi=(\varphi_k)_{k\in\mathbb{Z}}$ of linear maps $\varphi_k:U_k\to V_k$.

A \textbf{Lie superalgebra} is a vector space $\scr{G}$ together
with a \textbf{Lie superbracket} on $\scr{G}$, that is to say: a morphism
$[\cdot,\cdot] : \scr{G}\otimes \scr{G} \to \scr{G}$
satisfying antisymmetry and the Jacobi identity,
\begin{align}
[x,y]+(-1)^{xy}[y,x]&=0\,,\nn\\
[x,[y,z]]-(-1)^{xy}[y,[x,z]]&=[[x,y],z]\, \label{antisymandjacobi}
\end{align}
for all $x,y,z\in\scr G$ such that $x$ and $y$ are homogeneous with respect to the $\mathbb{Z}/2\mathbb{Z}$-grading.
Here (and below), when we write $xy$ in the exponent of a power of $(-1)$, we of course mean the product of the
degrees of $x$ and $y$. A morphism of Lie superalgebras from $\scr G$ to $\scr H$ is a
morphism of vector spaces $\varphi:\scr G\to \scr H$ commuting with the respective Lie superbrackets:
\begin{equation}
\varphi\big([x,y]_{\scr G}\big)=\big[\varphi(x),\varphi(y)\big]_{\scr H}.
\end{equation}

Given a Lie superalgebra $G$ and a subset $S$ of $G$, the intersection of all subalgebras
of $G$ including $S$ is the subalgebra of $G$ \textbf{generated} by $S$. If this intersection is $G$
itself, then $G$ is \textbf{generated} by $S$.

\subsection{Local Lie superalgebras}

\begin{defi}
A  vector space $U$ is \textbf{local} if 
$U_k=0$ for $k\notin\{-1,0,1\}$.
The \textbf{local part} $U^\Loc$ of a vector space $U$ is the sum of the subspaces $U_k$
for $k\in\{-1,0,1\}$.
A local morphism is a morphism of local vector spaces,
while the local
part $\varphi^\Loc$ of a morphism of vector spaces $\varphi$ is obtained by restricting the domain and codomain accordingly.
Conversely, $U$ and $\varphi$ are \textbf{global} extensions of $U^\Loc$ and $\varphi^\Loc$,
respectively.
A \textbf{local Lie superalgebra} is a local vector space $\scr{G}^\Loc$ together
with a \textbf{local Lie bracket} on $\scr{G}^\Loc$, that is, a local morphism
$[\cdot,\cdot] : (\scr{G}\otimes \scr{G})^\Loc \to \scr{G}^\Loc$,
satisfying (\ref{antisymandjacobi})
for all $x,y,z\in\scr{G}^\Loc$ such that $x$ and $y$ are homogeneous with respect to the $\mathbb{Z}/2\mathbb{Z}$-grading and all involved brackets are defined. 

Thus a local Lie superalgebra is in fact not an algebra since the bracket is not defined for a pair of
homogeneous elements
of the same degree $\pm1$.
We may talk about subalgebras of local Lie superalgebras
(possibly generated by some subset) in the same way as for Lie superalgebras.

It will be convenient below to generalise
the concept of local Lie superalgebras to \textbf{regional} Lie superalgebras (and related concepts
correspondingly)
by replacing the set $\{-1,0,1\}$ with any integer interval, which we then call the corresponding
\textbf{region}. We say that a regional Lie superalgebra is \textbf{semilocal} if the region is
$\{k\in\mathbb{Z}\,|\,k\leq 1\}$.
\end{defi}

\begin{proposition}\label{freely}$ $
\begin{enumerate}
\item[{\rm(a)}] Any local Lie superalgebra $\scr G^\Loc$
has a unique semilocal extension $\scr G_{1-}$
such that $\scr G_-$ is freely generated by $\scr G_{-1}$.
\item[{\rm(b)}]
Any semilocal Lie superalgebra $\scr G_{1-}$
has a unique global extension $\scr G$
such that $\scr G_{+}$ is freely generated by $\scr G_{1}$.
\end{enumerate}
\end{proposition}

\begin{proof}
Let $\scr G_-$ be the free Lie superalgebra generated by $\scr G_{-1}$.
Let $p\geq1$ and
suppose 
that $\scr G^\Loc$
has a unique regional extension with region $\{k\in\mathbb{Z}\,|\,-p\leq k\leq 1\}$,
which also extends the regional part of $\scr G_-$ with region $\{k\in\mathbb{Z}\,|\,-p\leq k\leq-1\}$.
We need to extend this regional Lie bracket further with a morphism
\begin{align} \label{bracketpart}
[\cdot,\cdot]\ :\ \scr G_{-p-1}\otimes (\scr G_0 \oplus \scr G_1) \to \scr G_{-p-1}\oplus \scr G_{-p}\,
\end{align}
such that
\begin{align}
[[x,y],u] = [x,[y,u]]-(-1)^{xy}[y,[x,u]] \label{regjacobi}
\end{align}
for any $x,y\in \scr G_{-}$ with $|x|+|y|=-p-1$ and $u\in\scr G_0 \oplus \scr G_1$.

Let $\tilde{\scr G}_{-p-1}$ 
be the direct sum of
$\scr G_{m} \wedge \scr G_{n}$ for all $m,n\leq -1$ such that $m+n=-p-1$
and define a morphism
\begin{align} \label{bracketpartprel}
\llbracket\cdot,\cdot\rrbracket\ :\ \tilde {\scr G}_{-p-1}\otimes (\scr G_0 \oplus \scr G_1) \to \scr G_{-p-1}\oplus \scr G_{-p}\,
\end{align}
by the condition
\begin{align}
\llbracket x \wedge y,u\rrbracket = [x,[y,u]]-(-1)^{xy}[y,[x,u]]
\end{align}
for $x \in \scr G_m$ and $y\in\scr G_n$, where $m,n\leq -1$ and $m+n=-p-1$.

Now there is an isomorphism $\tilde{\scr G}_{-p-1}/\scr J \to \scr G_{-p-1}$ 
where $\scr J$ is the subspace spanned by all elements
\begin{align}
x \wedge [y,z]-(-1)^{xy}v\wedge[x,z]-[x,y]\wedge z\,
\end{align}
such that $|x|,|y|,|z|\leq-1$ and $|x|+|y|+|z|=-p-1$.
If we can show that $\llbracket \scr J, \scr G_0 \oplus \scr G_1 \rrbracket=0$, then
the morphism (\ref{bracketpartprel}) induces a morphism (\ref{bracketpart})
such that the Jacobi identity (\ref{regjacobi}) is satisfied.
Indeed,
\begin{align}
\llbracket [x,y] \wedge z , u\rrbracket
&=[[x,y],[z,u]]-(-1)^{z(x+y)}[z,[[x,y],u]] \nn\\
&=[x,[y,[z,u]]]-(-1)^{z(x+y)}[z,[x,[y,u]]]\nn\\
&\quad\,-(-1)^{xy}[y,[x,[z,u]]]+(-1)^{z(x+y)+xy}[z,[y,[x,u]]]\nn\\
&=[x,[[y,z],u]]+(-1)^{yz}[[x,z],[y,u]]\nn\\
&\quad\,-(-1)^{xy}[y,[[x,z],u]]-(-1)^{x(y+z)}[[y,z],[x,u]]\,\nn\\
&=\llbracket x\wedge[y,z],u\rrbracket-(-1)^{xy}\llbracket y\wedge [x,z],u\rrbracket\,
\end{align}
for any $u\in \scr G_0 \oplus \scr G_1$.
Thus we can unambiguously \textit{define} the morphism (\ref{bracketpart})
by the Jacobi identity (\ref{regjacobi}), and the uniqueness then
follows since the brackets on the right hand side of (\ref{regjacobi})
are unique.
Having proven part (a) by induction,
part (b) can be proven in the same way, replacing $\scr G_0 \oplus \scr G_1$ with $\scr G_{0-}$
and $\scr G_-$ with $\scr G_+$.
\end{proof}

\begin{corollary} \label{maxexistence}
Any local Lie superalgebra $\scr G_{-1} \oplus \scr G_0 \oplus \scr G_1$
has a unique global extension $\scr G$
such that $\scr G_{\pm}$ is freely generated by $\scr G_{\pm1}$. 
\end{corollary}

\begin{proof}
This follows from part (a) and (b) of Proposition~\ref{freely} together.
\end{proof}
\noindent
An alternative proof of Corollary~\ref{maxexistence} (for Lie algebras)
is given in \cite{kacSimpleIrreducibleGraded1968} (as part of the proof of Proposition~4).

\subsection{Maximality and minimality}

\begin{defi}
A Lie superalgebra $\scr G$ is \textbf{locally generated} if 
it is generated by its local part $\scr G^\Loc$.
A locally generated Lie superalgebra
$\scr G$ is \textbf{maximal} (or a \textbf{maximal extension} of its local part)
if
any isomorphism  of local Lie superalgebras 
$\scr G^\Loc\to\scr H^\Loc$, where $\scr H$ is a locally generated Lie superalgebra,
extends to morphism $\scr G\to \scr H$.
It is \textbf{minimal} (or a \textbf{minimal extension} of its local part) if
any local isomorphism
$\scr H^\Loc\to\scr G^\Loc$, where $\scr H$ is a locally generated Lie superalgebra,
extends to a morphism $\scr H\to \scr G$.
\end{defi}

\noindent
The following lemma says that we can without loss of generality assume that the morphisms
$\scr G \to \scr H$ and $\scr H \to \scr G$ in the above definitions are surjective.

\begin{lemma} \label{surj}
Let $\scr G$ and $\scr H$ be locally generated Lie superalgebras  
and
let $\varphi\colon\scr G \to \scr H$ be a morphism whose local part is surjective. Then $\varphi$ itself is surjective too.
\end{lemma}
\begin{proof} 
Let $p\geq1$ and suppose $\varphi(\scr G_{\pm k})=\scr H_{\pm k}$ for any $k=1,\ldots,p$. Then any element 
in $\scr H_{\pm(p+1)}$
is a sum of terms
\begin{equation}
[\varphi(x_{\pm 1}),\varphi(y_{\pm p})]=\varphi[x_{\pm 1},y_{\pm p}]
\end{equation}
and the lemma follows by induction.
\end{proof}

\begin{proposition}\label{propminimal}
Maximal or minimal
extensions of a local Lie superalgebra are unique.
\end{proposition}

\begin{proof} 
Let $\scr G$ and $\scr H$ be both maximal or both minimal
extensions of a local Lie superalgebra $\scr G^\Loc$. Then there are
morphisms
$\varphi:\scr G\to\scr H$ and $\chi:\scr H\to\scr G$ whose local parts are identity maps.
According to Lemma~\ref{surj}, they are surjective. 
Similarly to the proof of Lemma~\ref{surj}, it then follows by induction that $\varphi$ 
is also injective, with inverse $\chi$.
Indeed, let $p\geq1$ and suppose $\chi\big(\varphi(x_{\pm k})\big)=x_{\pm k}$
for any $x\in \scr G$ and $k=1,\ldots,p$.
Then
\begin{align}
\chi\big(\varphi[x_{\pm 1},y_{\pm k}]\big)=\chi[\varphi(x_{\pm 1}),\varphi(y_{\pm k})]=[\chi\big(\varphi(x_{\pm 1})\big),
\chi\big(\varphi(y_{\pm k})\big)]=[x_{\pm 1},y_{\pm k}]
\end{align}
and by induction this holds for all $k\geq1$. Finally, it easily follows by induction that this isomorphism
must be the identity map.
\end{proof}

\begin{proposition}\label{charmaxmin} Let $\scr G^\Loc$ be a local Lie superalgebra. Let $\scr G$ be the 
unique global extension of $\scr G^\Loc$
such that $\scr G_{\pm}$ is freely generated by $\scr G_{\pm1}$
(which exists according to Propositions~\ref{freely} and \ref{propminimal}),
and let $\scr I$ be the sum
of all ideals of $\scr G$ intersecting $\scr G^\Loc$ trivially.
\begin{enumerate}
\item[{\rm(a)}] The Lie superalgebras $\scr G$ and $\scr G/\scr I$ are maximal and minimal
extensions of $\scr G^\Loc$, respectively. 
\item[{\rm(b)}] A locally generated Lie superalgebra $\scr H$
is minimal if and only if it has no non-trivial ideal intersecting
the local part trivially.
\end{enumerate}
\end{proposition}
\begin{proof}
Let $\scr H$ be a locally generated Lie superalgebra
with $\scr H^\Loc =\psi(\scr G^\Loc)$ for some local
Lie superalgebra isomorphism $\psi:\scr G^\Loc \to \scr H^\Loc$. Then,
since $\scr G_\pm$ is freely generated by $\scr G_{\pm1}$,
there are morphisms $\varphi_\pm:\scr G_\pm \to \scr H_\pm$
such that $\varphi_{\pm}(x_{\pm1})=\psi(x_{\pm1})$
for any $x\in\scr G$. We set $\varphi_0(x_0)=\psi(x_0)$ and let
$\varphi:\scr G \to\scr H$ be the linear map
given by $\varphi(x)=\varphi_\pm(x)$ for $x\in\scr G_\pm$ and $\varphi(x)=\varphi_0(x)=\psi(x)$ for $x\in\scr G_0$.
It is then easy to show by induction that $\varphi[x,y]=[\varphi(x),\varphi(y)]$ not only when
$x$ and $y$ belong to the same subalgebra $\scr G_\pm$ but also when $x\in\scr G_{0\pm}$ and $y \in\scr G_{0\mp}$.
Thus $\varphi$ is a morphism whose local part is $\psi$. This proves that $\scr G$ is maximal.

In order to show that $\scr G/\scr I$ is minimal, let $\kappa:\scr H^\Loc \to (\scr G/\scr I)^\Loc$ be a local
Lie superalgebra isomorphism. 
It induces a local
Lie superalgebra isomorphism $\scr H^\Loc \to \scr G^\Loc$
with inverse $\psi:\scr G^\Loc\to\scr H^\Loc$
such that $\kappa \circ \psi: \scr G^\Loc \to (\scr G/\scr I)^\Loc$ is the natural morphism.
Extend $\psi$ to $\varphi:\scr G \to\scr H$ as above. Then, since $\ker{\varphi}$ is included in
$\scr I$, there is a morphism
$\chi: \scr H \to \scr G/\scr I$ such that $\chi \circ \varphi : 
\scr G \to (\scr G/\scr I)$ is the natural morphism, and its local part is $\kappa$.
Thus $\scr G/\scr I$ is minimal and we have proven part (a) of the proposition.

Since $\scr I$ is the maximal ideal
of $\scr G$ intersecting the local part trivially, there is no non-trivial such ideal of $\scr G/\scr I$.
Furthermore, since this is the unique (according to Proposition~\ref{propminimal}) minimal extension 
of the local Lie superalgebra $(\scr G/\scr I)^\Loc\simeq \scr G^\Loc$, and $\scr G^\Loc$ was arbitrary,
there is no nontrivial such ideal of any
minimal Lie superalgebra.
Conversely, suppose $\scr H$ has no non-trivial ideal intersecting the local part trivially,
and let $\chi$ be a morphism $\chi:\scr H \to \scr G/\scr I$ such that the local part is an isomorphism.
Then $\ker(\chi)$ must be zero and $\chi$ is an isomorphism, so that $\scr H$ is minimal. This proves 
part (b) of the proposition.
\end{proof}

\subsection{Transitivity}

\begin{defi} Let $m,n$ be integers.
A Lie superalgebra $\scrG$ is \textbf{negatively $n$-transitive} if, for any 
$x\in\scr G_{n+}$, it follows from $[ \scrG_{-},x]=0$ that $x=0$.
It is \textbf{positively $m$-transitive}
if, for any 
$x\in\scr G_{m-}$,
it follows from $[ \scrG_{+},x]=0$ that $x=0$. It is $(m,n)$-\textbf{transitive}
if it is both positively $m$-transitive and negatively $n$-transitive.
\end{defi}

\begin{remark}
We are usually most interested in the cases where $m$ is non-positive and $n$ is non-negative, but the definitions
make sense and the results below are valid also otherwise.
\end{remark}

\begin{remark}
If $\scrG$ is locally generated, then
$[ \scrG_{\pm},x]=0$ is equivalent to $[ \scrG_{\pm1},x]=0$.
\end{remark}

\begin{remark}
The terms \emph{transitive} and \emph{bitransitive} usually means what we here call \emph{negatively $0$-transitive}
and \emph{$(0,0)$-transitive}, respectively, for example in \cite{kacLieSuperalgebras1977,scheunert}.
However, in \cite{kacSimpleIrreducibleGraded1968},
the term \emph{transitive} means \emph{$(0,0)$-transitive}. In \cite{Kantor},
a graded Lie superalgebra is said to be \emph{of type A} if it is negatively $0$-transitive
and its negative part is locally generated. It is said to be \emph{of type A$'$} if in addition
also its positive part is locally generated. We have chosen the terminology here to avoid any possible confusion,
and since the $\mathbb{Z}$-grading in our constructions is reversed compared to 
corresponding constructions in the literature.
\end{remark}
\noindent
The definition implies the following straightforward result:
\begin{lemma}\label{lemmabi}
Any $(m,n)$-transitive Lie superalgebra is also $(p,q)$-transitive, for any $p\leq m$ and $q\geq n$. 
\end{lemma}

\noindent From it, we have the following general result:
\begin{proposition}\label{propbi}
A Lie superalgebra $\scr G$
is $(m,n)$-transitive if and only if it 
it has no non-trivial ideal
included in $\scr G_{m-}$ or $\scr G_{n+}$.
\end{proposition}

\begin{proof} We prove the statement for
ideals included in $\scr G_{n+}$
as the same argument works for
ideals included in $\scr G_{m-}$.
Suppose that there is an ideal $\scr D$
included in $\scr G_{n+}$. Let $q\geq n$ be the smallest integer such that 
$\scr D_{q} \neq 0$. Then there is a nonzero $x \in \scr D_{q}$ such that $[\scrG_{-},x]=0$,
so $\scrG$ is not negatively $q$-transitive. By contraposition of Lemma \ref{lemmabi}, it is also
not negatively $n$-transitive. 

Conversely, for some $n$, suppose that $\scr G$
is not negatively $n$-transitive and
let $x$ be a nonzero element in $\scr G_{n+}$ such that 
$[\scrG_{-},x]=0$. Then it is easy to see that the $\scr G_{0+}$-module generated by $x$ is a non-trivial ideal of 
$\scr G$ included in $\scr G_{n+}$.
\end{proof}

\begin{corollary}\label{propbimin}
A locally generated Lie superalgebra $\scr G$ 
is minimal if and only if it is $(-2,2)$-transitive.
In particular, any $(0,0)$-transitive Lie superalgebra $\scr G$ is minimal.
\end{corollary}

\begin{proof}
This is an application of Proposition \ref{propbi} to a minimal Lie superalgebra, as characterised in Proposition~\ref{charmaxmin}.
The second statement then follows by Lemma \ref{lemmabi}, which says
that any $(0,0)$-transitive Lie superalgebra is also $(-2,2)$-transitive.\end{proof}

\noindent
The latter statement in Corollary \ref{propbimin} is part (a) of Proposition~5 in
\cite{kacSimpleIrreducibleGraded1968}.

\section{Kantor's construction and prolongation}\label{sec3}

\subsection{The universal graded Lie superalgebra associated to a vector space} \label{uglsa}

Everywhere in Section~\ref{uglsa}--\ref{redprol}, we let
$\scrU_1$ be a vector space which is concentrated in degree $1$,
and thereby odd.
We extend $\scrU_1$ to a semilocal vector space
$\scrU_{1-}$ recursively by
\begin{align} \label{defUsubspaces}
\scrU_{-p+1}=\Hom(\scrU_1,\scrU_{-p+2})\simeq \Hom(({\scrU}_1)^{\otimes p},{\scrU}_1)\,
\end{align}
for $p=1,2,\ldots$. In particular, $\scrU_0=\End\,{\scrU_1}$.

\begin{proposition}\label{prop:semilocal}
The vector space $\scr U_{1-}$ together with the 
morphism
\begin{align}
[\cdot,\cdot]\ :\ (\scr U_{1-}\otimes \scr U_{1-})_{1-} \to \scr U_{1-}
\end{align}
defined recursively by 
\begin{align} \label{recbase}
[x, u]  &= x(u)\,, &[u,x]  &=- (-1)^xx(u)\,,
&[x , y](u) &= [x ,y(u)]
+ (-1)^{y} [x(u) , y]
\end{align}
for any $x,y\in \scr U_{0-}$ and $u\in \scr U_1$,
is a semilocal Lie superalgebra.
\end{proposition}
\noindent
In particular, $U_0=\mathfrak{gl}(U_1)$ as a Lie algebra.

\begin{proof}
Suppose that the subspace
$\scr U_{-p}\oplus \cdots \oplus \scr U_0 \oplus \scr U_1 $
is a regional Lie superalgebra for some $p\geq1$.
We then have
\begin{align}
[[x,y],z](u)&=[[x,y],z(u)]+(-1)^z[[x,y(u)],z]
+(-1)^{y+z}[[x(u),y],z]\nn\\
&=[x,[y,z(u)]]+(-1)^z[x,[y(u),z]]
+(-1)^{y+z}[x(u),[y,z]]\nn\\
&\quad\,-(-1)^{xy}[y,[x,z(u)]]-(-1)^{x(y+1)+z}[y(u),[x,z]]
-(-1)^{xy+z}[y,[x(u),z]]\nn\\
&=[x,[y,z]](u)-(-1)^{xy}[y,[x,z]](u)
\end{align}
for any $u\in \scr U_1$ and $x,y,z\in\scr G_-$ with $|x|+|y|+|z|=-p-1$. Thus the Jacobi identity
\begin{align} \label{jacobi1}
[x,[y,z]]-(-1)^{xy}[y,[x,z]]&=[[x,y],z]
\end{align}
is satisfied, and the proposition follows by induction.
\end{proof}
\noindent
We now let $\scr U$
be the unique global extension of the semilocal Lie algebra $\scr U_{1-}$ such that $\scr U_{+}$ is freely generated by
$\scr U_1$ (which exists according to Proposition~\ref{freely}).
It follows directly by the construction that $\scr U$ is a positively $0$-transitive Lie superalgebra.
The subalgebra $\scr U_+$ is generated by $\scr U_1$, but $\scr U_-$ is not generated by $\scr U_{-1}$,
so $\scr U$ is not locally generated.

\begin{defi}\label{defKantor}
Following Kantor \cite{Kantor}, we call $\scr U$ the \textbf{universal (graded) Lie superalgebra} associated to the
vector space $\scr U_1$. (The terminology will be justified by Theorem~\ref{Kantor-theorem} below.)
\end{defi}

\noindent
We here follow the approach of \cite{Palmkvist:2009qq,Palmkvist:2020hga}, where
the original concept of Kantor has been adapted to Lie superalgebras and given an
(equivalent) recursive definition.

\subsection{Prolongation}\label{prolsubsec}

\noindent
For any Lie superalgebra $\scr G$ with $\scr G_1=\scr U_1$, we let $\rho:\scr G_{1-} \to \scr U_{1-}$ be 
the vector space morphism
recursively defined by $\rho(u)=u$ for $u\in\scr G_1=\scr U_1$ and 
$\rho(x)=\rho \circ (\ad{\,x})$ for
$x\in \scr U_{0-}$, so that the recursion relation reads
$\rho(x)(u)=\rho[x,u]$ for any $u\in\scr G_1=\scr U_1$.

In particular, the restriction of $\rho$ to $G_0 \to U_0=\mathfrak{gl}(G_1)$
is the representation corresponding to the $G_0$-module structure on $G_1$ that the adjoint action
gives rise to, so it makes sense to use the notation $\rho$ also in this more general meaning.

\begin{proposition}\label{rhoprop}$ $
\begin{itemize}
\item[{\rm(a)}] If $\rho$ is injective, then $\scr G$ is positively $0$-transitive.
\item[{\rm(b)}] If $\scr G$ is positively $0$-transitive and $\scr G_+$ is generated by $\scr G_1$, then $\rho$ is injective.
\item[{\rm(c)}] The vector space morphism $\rho$ is a semilocal Lie superalgebra morphism.
\end{itemize}
\end{proposition}

\begin{proof}$ $
\begin{itemize}
\item[(a)] If $\rho$ is injective and $[\scr G_+,x]=0$ for some $x\in\scr G_{0-}$, 
then $[x,u]=0$ and $\rho(x)(u)=\rho[x,u]=0$ for all $u\in\scr G_1$. Since $\rho$ is injective,  $x=0$.
\item[(b)]
Suppose that the regional part of $\rho$ with region $\{-p,\ldots,1\}$, for some $p\geq-1$,
is injective. If $\rho(z)=0$ for $z\in{\scr G}_{-p-1}$, then $0=\rho(z)(u)=\rho[z,u]$ 
and
$[z,u]=0$ for all $u\in\scr G_1$. Since
$\scr{G}_+$ is generated by $\scr G_1$ and $G$ is positively $0$-transitive,
$z=0$. It thus follows by induction that $\rho$ is injective.
\item[(c)] For any $z\in \scr G_{0-}$ and $u\in \scr G_1=\scr U_1$, we have
\begin{align}
[\rho(z),\rho(u)]=[\rho(z),u]=\rho(z)(u)=\rho[z,u]\,. \label{base}
\end{align}
Suppose now that the regional part of $\rho$ with region $\{-p,\ldots,1\}$, for some $p\geq-1$,
is a regional Lie superalgebra morphism.
Then
\begin{align}
\rho[x,y](u)
&=\rho[[x,y],u]\nn\\
&=\rho[x,[y,u]]-(-1)^{xy}\rho[y,[x,u]]\nn\\
&=[\rho(x),\rho[y,u]]-(-1)^{xy}[\rho(y),\rho[x,u]]\nn\\
&=[\rho(x),\rho(y)(u)]-(-1)^{xy}[\rho(y),\rho(x)(u)]\nn\\
&=[\rho(x),\rho(y)](u)
\end{align}
for $x,y\in \scr G_{0-}$ with $|x|+|y|=-p-1$ and $u\in \scr G_1=\scr U_1$. Thus $[\rho(x),\rho(y)]=\rho[x,y]$.
The statement now follows by induction, with (\ref{base}) as the base case.
\end{itemize}
\end{proof}
\begin{corollary}
For any Lie superalgebra $\scr G$ with $\scr G_1=\scr U_1$, the subspace $\rho(\scr G_{0-})$ of $\scr U_{0-}$
is a subalgebra. If in addition $\scr G$ is positively $0$-transitive and $\scr G_+$ is generated by $\scr G_1$,
then this subalgebra is isomorphic to $\scr G_{0-}$.
\end{corollary}

\noindent
The following theorem, given in \cite{Kantor} (Theorem~2),
describes a universal property of~$\scr U$.

\begin{theorem} \label{Kantor-theorem}
Let $\scr G$ be a 
positively $0$-transitive Lie superalgebra $\scr G$ with positive part $\scr G_{+}$ 
generated by $\scr G_1=\scr U_1$,
so that $\scr U$ is the universal Lie superalgebra associated to $\scr G_1$.
\begin{itemize}
\item[{\rm(a)}] The Lie superalgebra $\scr G$ is isomorphic to
a quotient
$({\scr T}_{0-}
\oplus {\scr U}_+ )/\scr D_{2+}\,$,
where 
${\scr T}_{0-}$ is a subalgebra of $\scr U_{0-}$ and 
$\scr D_{2+}$ is an ideal of
${\scr T}_{0-}\oplus {\scr U}_+ $ included in ${\scr U}_{2+}$.
\item[{\rm(b)}] Conversely, 
for any subalgebra $\scr T_{0-}$ of $\scr U_{0-}$ 
and ideal $\scr D_{2+}$ of
${\scr T}_{0-}
\oplus {\scr U}_+ $ included in ${\scr U}_{2+}$, the quotient
$({\scr T}_{0-}
\oplus {\scr U}_+ )/\scr D_{2+}$
is a positively $0$-transitive Lie superalgebra with positive part
generated by its subspace at $\mathbb{Z}$-degree $1$.
\end{itemize}
\end{theorem}
\begin{remark}
It is obvious that  \textit{as a vector space} $\scr G$ is isomorphic to $(\rho(\scr G_{0-})
\oplus {\scr U}_+ )/\scr D_{2+}\,$,
where~$\scr D_{2+}$
is an ideal of
${\scr U}_+ $ included in ${\scr U}_{2+}$. Part (a) of the theorem says that $\scr D_{2+}$ furthermore
is an ideal
of $\rho(\scr G_{0-})\oplus {\scr U}_+ $, not only of ${\scr U}_+$.
\end{remark}

\begin{proof} We will show that (a) holds with $T_{0-}=\rho(G_0)$.
Since $\scr{G}_+$ is generated by the subspace $\scr{G}_1$, there is
a surjective Lie superalgebra morphism
$\varphi_{+}: {\scr U}_+ \to \scr{G}_+ $
which is the identity map on $\scr U_{1}=\scr G_1$. Let $\varphi$ be the surjective
vector space morphism
\begin{align}
\varphi : \rho(\scr G_{0-})\oplus \scr U_{\,+} \to \scr G
\end{align}
given by $\varphi(\rho(x))=x$ for $x \in \scr G_{0-}$ and $\varphi(x)=\varphi_{\,+}(x)$ for $x \in \scr U_{+}$.
It is well defined since $\rho$ is injective by Proposition~\ref{rhoprop},
and accordingly, the restriction of $\varphi$ to $\scr G_{0-}$ is injective too.
It remains to show that 
\begin{align}
\varphi[\rho(z),v]=[\varphi(\rho(z)),\varphi(v)]
\end{align}
for all $z\in\scr G_{0-}$ and $v\in\scr U_+$. This can be done by induction over $|v|\geq1$ with the base case
\begin{align}
\varphi[\rho(z),u]=\varphi(\rho[z,u])=[z,u]=[\varphi(\rho(z)),\varphi(\rho(u))]
\end{align}
for $u\in\scr G_1$, obtained by acting with $\varphi$ on (\ref{base}),
and the induction step
\begin{align}
[\varphi(x),\varphi[v,w]]&=[\varphi(x),[\varphi(v),\varphi(w)]]\nn\\
&=[[\varphi(x),\varphi(v)],\varphi(w)]-(-1)^{vw}[[\varphi(x),\varphi(w)],\varphi(v)]\nn\\
&=[\varphi[x,v],\varphi(w)]-(-1)^{vw}[\varphi[x,w],\varphi(v)]\nn\\
&=\varphi[[x,v],w]-(-1)^{vw}\varphi[[x,w],v]=\varphi[x,[v,w]]\,
\end{align}
where $v,w\in\scr U_{+}$.
If we now let $\scr D_{2+}$ be the kernel of $\varphi$ (which as an ideal of $\scr U_{+}$ is equal to
the kernel of $\varphi_+$), then it follows that $\scr G$ is isomorphic to
the quotient
$\big(\rho(\scr G_{0-})\oplus \scr U_{+}\big)/\scr D_{2+}$. This proves part (a).

Conversely, consider any quotient $({\scr T}_{0-}\oplus {\scr U}_+ )/\scr D_{2+}$, where 
${\scr T}_{0-}\oplus {\scr U}_+$ is a subalgebra of $\scr U$ and $\scr D_{2+}$ is an ideal of
${\scr T}_{0-}
\oplus {\scr U}_+ $ included in ${\scr U}_{2+}$. The positive part is 
${\scr U}_+ /\scr D_{2+}$,
which is generated by its subspace at degree $1$,
and the non-positive part is embedded in $\scr U_{0-}$.
Since $\scr U_{0-}$ is positively $0$-transitive by construction, this also holds for 
the non-positive part of $({\scr T}_{0-}\oplus {\scr U}_+ )/\scr D_{2+}$, and thereby for the entire Lie superalgebra.
This proves part (b).
\end{proof}
\noindent
In the theorem, ${\scr T}_{0-}\oplus {\scr U}_+$ is a subalgebra of $\scr U$, which includes a subalgebra
$\scr T_{\scr 0-}$ of $\scr U_{0-}$ and the whole of $\scr U_+$. The next proposition characterises such subalgebras of 
$\scr U$.

\begin{proposition} \label{subalgebra}
Set $\scr T_1=\scr U_1$ and let $\scr T_{0-}$ be a subspace of $\scr U_{0-}$ such that
\begin{align} \label{subalgcond}
\scr T_{-k}\subseteq \Hom{(\scr U_1,\scr T_{-k+1})}
\end{align}
for $k=0,1,\ldots$. 
Then the subspace
${\scr T}_{0-}\oplus {\scr U}_+$ of $\scr U$
is a subalgebra. Conversely, the subspaces
$\scr T_{-k}$ for $k=0,1,\ldots$ of
any subalgebra ${\scr T}_{0-}\oplus {\scr U}_+$ of $\scr U$ satisfy this condition.
\end{proposition}

\begin{proof}
Let $T_{0-}$ be an arbitrary subspace of $U_{0-}$. It is easy to see that 
$[T_{0-},U_+]\subseteq T_{0-}\oplus U_+$ if and only if (\ref{subalgcond}) holds for
$k=0,1,\ldots$. It is also easy to see that if (\ref{subalgcond}) holds for
$k=0,1,\ldots$, then $[T_{0-},T_{0-}]\subseteq T_{0-}\oplus U_+$.
\end{proof}

In order to furthermore characterise subalgebras ${\scr T}_{0-}\oplus {\scr U}_+$
of $\scr U$ such that an ideal $\scr D_{2+}$ of $\scr U_+$
also is an ideal of ${\scr T}_{0-}\oplus {\scr U}_+$, we of course have to add the condition 
$[\scr T_{0-},\scr D_{2+}]\subseteq \scr D_{2+}$
to (\ref{subalgcond}). This condition means that $\scr T_{0-}$ be included in the \textbf{idealiser} 
(or \textbf{normaliser}) $\scr N$ of
$\scr D_{2+}$ in $\scr U$, defined by
\begin{align}
\scr N =\{x \in \scr U \,|\, [x,d]\in \scr D_{2+} \text{ for all } d\in \scr D_{2+}\}\,.
\end{align}
Obviously, $\scr N_+=\scr U_+$, since $\scr D_{2+}$ is already an ideal of $\scr U_+$.
By extending ${\scr T}_{0-}
\oplus {\scr U}_+ $ to $\scr N$ we thus get an extension of $({\scr T}_{0-}
\oplus {\scr U}_+ )/\scr D_{2+}$ to $\scr N/\scr D_{2+}$, 
where the positive part is $\scr U_+/\scr D_{2+}$ and the non-positive part $\scr N_{0-}/\scr D_{2+}\simeq \scr N_{0-}$
is in some sense as
large as possible (under the condition of respecting the ideal $\scr D_{2+}$).
This motivates the following definition.

\begin{defi}\label{prolong}
Let $\scr G_+$ be a Lie superalgebra generated by $\scr G_1=\scr U_1$, so that
$\scr G_+\simeq \scr U_+/\scr D_{2+}$ for some ideal $\scr D_{2+}$ of
$\scr U_+$ included in $\scr U_{2+}$. Let $\scr N$ be the 
idealiser of $\scr D_{2+}$ in $\scr U$.
Then the \textbf{prolongation} of
$\scr G_+$ is the quotient
${\scr N} / \scr D_{2+}$.
\end{defi}

\noindent
The prolongation of a locally generated Lie superalgebra $G_+$ is usually (adapted to our conventions)
defined as the maximal (with respect to inclusion)
positively $0$-transitive Lie superalgebra $\hat{G}$ such that $\hat{G}_+=G_+$.
The existence and uniqueness of a prolongation in this definition
follows from Theorem~\ref{Kantor-theorem}, and the two definitions are equivalent up to the isomorphism
$\scr G_+\simeq \scr U_+/\scr D_{2+}$. We get the
following version of Theorem~\ref{Kantor-theorem}.

\begin{proposition}~\label{prolongation} 
Any positively $0$-transitive Lie algebra $\scr G$
such that $\scr G_+$
is generated by
$\scr G_1$ is
embedded in the prolongation
of $\scr G_+$.
\end{proposition}

\begin{proof}
According to Theorem~\ref{Kantor-theorem}, the Lie superalgebra
$\scr G$ is (with $\scr G_1=\scr U_1$) isomorphic to $(\scr{T}_{0-}
\oplus \scr{U}_+ )/\scr D_{2+}$, where 
$\scr{T}_{0-}$ is a subalgebra of $\scr U_{0-}$ and $\scr D_{2+}$ is an ideal of
$\scr{T}_{0-}
\oplus \scr{U}_+ $ included in $\scr{U}_{2+}$. This means that $\scr{T}_{0-}
\oplus \scr{U}_+ $ is a subalgebra of the idealiser $ \scr N$ of $\scr D_{2+}$, 
when $\scr D_{2+}$ is considered as a subspace of $\scr U$.
Thus $(\scr{T}_{0-}
\oplus \scr{U}_+ )/\scr D$ is a subalgebra of $ \scr N/\scr D_{2+}$, that is, of the prolongation of $\scr G_+$.
\end{proof}

\begin{remark}
In the original definition of a prolongation 
by Tanaka \cite{tanaka} (generalised from Lie algebras
to Lie superalgebras and adapted to our conventions), the Lie superalgebra $G_+$
was assumed to be finite-dimensional. As we have seen, this is not necessary. On the other hand,
in the case where $G_+$ is a finite-dimensional Lie algebra but not necessarily generated by $G_1$,
the existence and uniqueness
of a corresponding extension, called \textit{universal graded Lie algebra of type $G_+$}
was proven in \cite{qingyunguangyu}.
\end{remark}

\begin{example}
Let $U_1$ be an $n$-dimensional odd vector space with basis elements $E_a$, where $a=1,\ldots,n$.
Then, as vector spaces, $U_0\simeq (U_1)^\ast \otimes U_1$, and
\begin{align}
U_{-1}\simeq (U_1)^\ast \otimes U_0 \simeq (U_1)^\ast \otimes (U_1)^\ast \otimes U_1\,.
\end{align}
Accordingly, we introduce basis elements $K^a{}_b$ for $U_0$ and $K^{ab}{}_c$ for $U_{-1}$
(where all indices take the values $1,\ldots,n$ and no symmetry or antisymmetry among them is assumed) with
relations
\begin{align}
[E_a,K^b{}_c]&=\delta_a{}^b E_c\,, & [E_a,K^{bc}{}_d]&=\delta_a{}^b K^c{}_d\,.
\end{align}
Let $D_{2+}=U_{2+}$ (so that $U_+/D_{2+}\simeq U_1$, considered as an abelian odd Lie superalgebra) and let $N$ be the idealiser of $D_{2+}$ in $U$. Then
$N_{0+}=U_{0+}$. In order to determine the negative part of $N$, let us first consider $N_{-1}$
and decompose $U_{-1}$ into a direct sum of two subspaces with bases $K^{(ab)}{}_c$ and $K^{[ab]}{}_c$,
respectively, where
\begin{align}
K^{(ab)}{}_c &= \tfrac12(K^{ab}{}_c+K^{ba}{}_c)\,, & K^{[ab]}{}_c &= \tfrac12(K^{ab}{}_c-K^{ba}{}_c)\,.
\end{align}
The subspace $N_{-1}$ of $U_{-1}$ consists
of all elements $x \in U_{-1}$ such that $[x,U_{2}]=0$.
We have
\begin{align}
[K^{ab}{}_c,[E_d,E_e]]=2[[K^{ab}{}_c,E_{[d}],E_{e]}]=-2\delta_{[d}{}^a[K^b{}_c,E_{e]}]=
2\delta_{[d}{}^{a}\delta_{e]}{}^b E_c
\end{align}
and it follows that $[K^{(ab)}{}_{c},U_2]=0$ whereas $[K^{[ab]}{}_{c},U_2]\neq0$.

Thus 
$N_{-1}\simeq(U_1{}^\ast \vee U_1{}^\ast ) \otimes U_1$,
and
generally, we will find
$N_{-p+1}\simeq(U_1{}^\ast)^{\vee p} \otimes U_1$ 
for all $p=0,1,\ldots,n$, where we use the symbol $\vee$ for graded symmetric tensor product,
\begin{align}
2(x\vee y) =x \otimes y-(-1)^{xy}y\otimes x.
\end{align}
Since $U_1{}^\ast$ is odd, this graded symmetry actually
means antisymmetry
and $N_{-p+1}=0$ for $p\geq n+1$. Since $D_{2+}=U_{2+}$, factoring out $D_{2+}$ does not affect the isomorphisms
$N_{-p+1}\simeq(U_1{}^\ast)^{\vee p} \otimes U_1$
for all $p=0,1,\ldots,n$, but sets all
subspaces at other degrees to zero, so we have
$(N/D_{2+})_{-p+1}\simeq(U_1{}^\ast)^{\vee p} \otimes U_1$ for $p=0,1,\ldots,n$
and $(N/D_{2+})_{2+}=0$.

The Lie superalgebra $N/D_{2+}$ obtained by prolongation in this example is 
denoted $W(n)$ and, in the classification of finite-dimensional simple Lie superalgebras over $\mathbb{C}$,
it is said to be of Cartan type. It is the derivation algebra of the associative and graded commutative
Grassmann superalgebra generated by 
$U_1$, and can also be considered as consisting of (odd) polynomial vector fields on $U_1$ \cite{kacLieSuperalgebras1977}. 
\label{W(n)}
\end{example}

\noindent
The procedure of prolongation originates from work by Cartan \cite{cartan}
and has since then appeared in many different versions,
in particular the definition by Tanaka that we referred to above, but also
the construction by Weisfeiler based on filtrations \cite{weisfeiler}. Certain reduced prolongations 
(see below)
are also known as \textit{partial} prolongations
and were described in \cite{ALShch,shch}.
We refer to \cite{leites2024classificationsimplecomplexlie} for more about the history of prolongations, and to 
\cite{kruglikov2022} for the geometric counterpart of the algebraic concept that we consider here.

\subsection{Reduced prolongation} \label{redprol}

The result in the preceding subsection can be generalised. Between the two extremal cases in Theorem~\ref{Kantor-theorem}
(where the subalgebra of $\scr U_{0-}$
is as small as possible, so that the embedding is an isomorphism) and Proposition~\ref{prolongation}
(where the subalgebra
is as large as possible), we can let the subspaces $\scr T_{-k}$ be arbitrary down to some degree $-p$, and 
then as large as possible for lower degrees.

\begin{proposition}\label{refinement}
Let $\scr D_{2+}$ be an ideal of $\scr U_+$ and let $\scr N$ be the idealiser of $\scr D_{2+}$ in $\scr U$.
Set $\scr T_1=\scr U_1$ and, for some $p\geq0$, 
let $\scr T_0,\scr T_{-1},\ldots,\scr T_{-p}$ be subspaces of ${\scr U}_{\,0},\scr U_{\,-1},\ldots,\scr U_{\,-p}$ satisfying
the conditions
\begin{align} \label{V-villkor}
\scr T_{-k} &\subseteq \Hom{(\scr U_1,\scr T_{-k+1})}\,, &
[\scr T_{-k},\scr D_{2+}]&\subseteq \scr D_{2+}\,.
\end{align}
Furthermore, set $\scr N'{}_{-p}=\scr T_{-p}$ and define recursively
\begin{align}
\scr N'{}_{-k}=\Hom{(\scr U_1,\scr N'{}_{-k+1})} \cap \scr N_{-k}
\end{align}
for all $k\geq p+1$.
Then the subspace 
\begin{align}
\scr N' = \cdots \oplus \scr N'{}_{-p-1} \oplus \scr T_{-p}
\cdots \oplus \scr T_{-1} \oplus \scr T_{0} 
\oplus \scr U_1 \oplus \scr U_2 \oplus \cdots
\end{align}
of $\scr U$ is a subalgebra of $\scr N$.
\end{proposition}
\begin{proof}
Proposition~\ref{subalgebra} says that $\scr N'$ is a subalgebra of $\scr U$, and the condition
$[\scr T_{-k},\scr D_{2+}]\subseteq \scr D_{2+}$ for $k\geq0$ says that $\scr N'$ is a subset of $\scr N$.
\end{proof}

\begin{defi}
As in Definition~\ref{prolong}, let $\scr G_+$ be a Lie algebra 
generated by $\scr G_1 = \scr U_1$,
so that $\scr G_+ \simeq \scr U_+/\scr D_{2+}$ for some ideal $\scr D_{2+}$ of $\scr U_+$ included in $\scr U_{2+}$. 
Let $\scr T_0,\scr T_{-1},\ldots$ be subspaces of $\scr U_0,\scr U_{-1},\ldots$
satisfying the
conditions (\ref{V-villkor}), and define $\scr N'$
from them
as in Proposition~\ref{refinement}. Then the \textbf{reduced prolongation} of
$\scr G_+$ with respect to $\scr T_{-p}\oplus \cdots \oplus \scr T_{-1} \oplus \scr T_{0}$ is the quotient
${\scr N'}/ \scr D$.
\end{defi}

\begin{proposition}\label{314}
Let $\scr G$ be a positively $0$-transitive Lie superalgebra
such that $\scr G_+$
is generated by
$\scr G_1=\scr U_1$, 
and let $\scr D_{2+}\subseteq \scr U_{2+}$ be an ideal of $\scr U_+$ such that $\scr G_+ \simeq \scr U_+/\scr D_{2+}$.
Let $\scr T_0,\scr T_{-1},\ldots,\scr T_{-p}$ be subspaces of $\scr U_0,\scr U_{-1},\ldots,\scr U_{-p}$
satisfying the
conditions
\begin{align}
\rho(\scr G_{-k})\subseteq
\scr T_{-k} &\subseteq \Hom{(\scr U_1,\scr T_{-k+1})}\,, &
[\scr T_{-k},\scr D_{2+}]&\subseteq \scr D_{2+}\,
\end{align}
for some $p\geq0$ and $0\leq k\leq p$, where $\rho:\scr G_{1-} \to \scr U_{\,1-}$ is the injective semilocal Lie superalgebra morphism defined above.
Then 
the Lie superalgebra $\scr G$ is embedded in 
the reduced prolongation
of $\scr G_+$ with respect to $\scr T_{-p}\oplus \cdots \oplus \scr T_{-1} \oplus \scr T_{0}$.
\end{proposition}
\begin{proof}
Again according to Theorem~\ref{Kantor-theorem}, the Lie superalgebra
$\scr G$ is (with $\scr G_1=\scr U_1$) isomorphic to $\big(\rho(\scr G_{0-})
\oplus \scr{U}_+ \big)/\scr D_{2+}$. 
In this case $\rho(\scr G_{0-})
\oplus \scr {U}_+ $ is a subalgebra not only of $\scr N$ but also of $\scr N'$.
Thus $\big(\rho(\scr G_{0-})
\oplus \scr {U}_+ \big)/\scr D_{2+}$ is a subalgebra of $ \scr N'/\scr D_{2+}$, that is, of the reduced
prolongation of $\scr G_+$
with respect to $\scr T_{-p}\oplus \cdots \oplus \scr T_{-1} \oplus \scr T_{0}$.
\end{proof}

\begin{example}
As in Example~\ref{W(n)}, we can set $D_{2+}=U_{2+}$ and consider $U_1$ as an abelian odd Lie superalgebra, 
but also equip $U_1$
with the structure of a $\fg$-module for some Lie algebra $\fg$ so that there is a corresponding
representation $\rho: \fg \to \End\,U_1$. We may then consider the reduced prolongation of $U_1/D_{2+}$
with respect to $\rho(\fg)$ at degree $0$.
More about this type of reduced prolongation (in the Lie algebra case)
can be found in \cite{kobayashi}. In the next section we will
continue to study the case where $U_1$ is a module
over a Lie algebra $\fg$.

\end{example}

\subsection{Embedding tensors and tensor hierarchy algebras}

Consider the local part of $U$ with basis elements $K^{ab}{}_c,K^a{}_b,E_a$ at degree $-1,0,1$, respectively,
as in Example~\ref{W(n)}. The local part of $W(n)$ is a subalgebra of this local Lie superalgebra, generated by
only the antisymmetric basis elements $K^{[ab]}{}_c$ at degree $-1$ and all basis elements $E_a$ at degree $1$.
We then also get back all $K^a{}_b$ at degree $0$ from brackets $[K^{[ab]}{}_c,E_d]$.
Instead of constructing $W(n)$ by prolongation, as in Example~\ref{W(n)}, we can now construct $W(n)$
as the minimal extension of this 
local Lie superalgebra.

This can be generalised. Instead of $D_{2+}=U_{2+}$
we can let $D_{2+}$ be any other ideal of
$U$ included in $U_{2+}$. Any such ideal $D_{2+}$ will 
determine a corresponding subspace $N_{-1}$ of $U_{-1}$, and by prolongation we then get
a Lie superalgebra with local part $N_{-1}\oplus N_0\oplus U_1$ generated by $N_{-1}$ and $U_1$.
However, 
in general this Lie superalgebra will not be locally generated and thus different from the minimal extension
of its local part, although the
two constructions happen to give the same result in the case of $W(n)$.
In \cite{palmkvistTensorHierarchyAlgebra2014}, generalisations of $W(n)$ were not constructed 
by prolongation
but as
as minimal extensions of their local parts, based on the following definition. 

\begin{defi}\label{defP}
Let $\scr T_{-1}$ be a subspace of $\scr U_{-1}$.
We define $\tilde{\scr P}(\scr U_{1},\scr T_{-1})$ as the subalgebra
of $\scr U$ generated by $\scr U_1$ and $\scr T_{-1}$. Let $\scr I_{2+}$ be the maximal ideal of
$\tilde{\scr P}(\scr U_{1},\scr T_{-1})$
intersecting the local part trivially, thus contained in $\scr U_{2+}$. We define 
the Lie superalgebra ${\scr P}(\scr U_{1},\scr T_{-1})$
as the quotient ${\scr P}(\scr U_{1},\scr T_{-1})=\tilde{\scr P}(\scr U_{1},\scr T_{-1})/\scr I_{2+}$.
\end{defi}

\begin{remark}
In \cite{palmkvistTensorHierarchyAlgebra2014}, the Lie superalgebra
$\tilde{\scr P}(\scr U_{1},\scr T_{-1})$ was denoted by $\mathcal{V}(\scr U_{1},\scr T_{-1})$
and ${\scr P}(\scr U_{1},\scr T_{-1})$ was denoted by $\mathcal{V}'(\scr U_{1},\scr T_{-1})$. 
\end{remark}

\noindent
Suppose now that the vector space $\scr U_1$ is a module 
over some Lie algebra $\fg$, so that there is
a representation $\rho:\fg \to \mathfrak{gl}(U_1)=\End\,{\scr U_1}$, and any element $x \in\fg$ acts on 
any element $u\in \scr U_1$ by $x\cdot u = \rho(x)(u)$. 
Then $\rho(\fg)$ is a subspace of $\End\,\scr U_1$ and
we can consider the reduced prolongation of $U_+$ with respect to $\rho(\fg)$ at degree $0$. 
Its subspace at degree $-1$ is 
\begin{align}
N_{-1}=\Hom{(\scr U_1,\rho(\fg))}\subseteq\Hom{(\scr U_1,\End\,{\scr U_1})}=\scr U_{-1}\,.
\end{align}
It follows that also $N_{-1}$ is a $\fg$-module, with the action of
any element $x \in\fg$ given by the adjoint action of $\rho(x)\in U_1$ in the Lie superalgebra $U$.
We can now restrict $N_{-1}$ further to a submodule $T_{-1}$
and factor out the maximal ideal $I_{2+}$ of $\tilde P(U_1,T_{-1})$ intersecting the local part trivially
in order to obtain $P(U_1,T_{-1})$.

Instead of specifying $T_{-1}$ we can specify an ideal $D_{2+}$ of $U$ included in $U_{2+}$
and let $T_{-1}\oplus \rho(\fg) \oplus U_1$ be the local part of the reduced prolongation of 
$U_+/D_{2+}$ with respect to $\rho(\fg)$, in the same way as for the full prolongation
that we considered above.

If $\fg$ is a complex finite-dimensional simple Lie algebra and $\scr U_1$ is a
finite-dimensional irreducible $\fg$-module with highest weight $\lambda$,
then there will always be an irreducible $\fg$-module included in $\scr U_2$ with highest weight $2\lambda$.
If we let $\scr D_{2+}$ be the ideal of $\scr U_+$ 
generated by this $\fg$-module, 
and let $T_{-1}$ be the subspace of $U_{-1}$ defined from $D_{2+}$ by reduced prolongation as explained above,
then the Lie superalgebra $P(U_1,T_{-1})$ 
is called a \textbf{tensor hierarchy algebra} \cite{palmkvistTensorHierarchyAlgebra2014}.
The reason for this terminology is that the condition $[\rho\circ\Theta,\scr D_{2+}]=0$
then turns out to be equivalent to 
the \textit{representation
constraint} in gauged supergravity theories, with global symmetry algebra $\fg$ and one-form gauge potentials transforming
in the representation $\rho$
corresponding to $\scr U_1$. Gauge covariance then requires the inclusion of two-forms
in the theory, corresponding to the subspace at degree $2$, and leads further to a \textit{tensor hierarchy} of gauge parameters, potentials and field strengths \cite{dewitGaugedSupergravitiesTensor2008}.
A linear map $\Theta :\scr U_1\to\fg$ which in addition to
the representation
constraint $[\rho\circ\Theta,\scr D_{2+}]=0$ also satisfies the \textit{quadratic constraint}
\begin{align} \label{qcon}
[\rho\circ\Theta,\rho\circ\Theta]=0
\end{align}
is in this context called an \textit{embedding tensor}.
As declared in the introduction, we will not go further into the representation constraint in the present paper.
Instead, we will in the next section look closer at the quadratic constraint (\ref{qcon}).

\section{Relationship with the Lie superalgebra induced by a Lie--Leibniz triple}\label{sec4}

This section is devoted to elaborate on the notion of an embedding tensor and the construction of
a Lie superalgebra that is canonically induced by it. 
At the core of this construction is a triple made of a Lie algebra $\fg$, a $\fg$-module $V$ 
with representation map $\rho:\fg\to\mathfrak{gl}(V)$ and 
and a linear map $\Theta:\scr V\to\fg$. 
We will keep this notation throughout the section.
We will refer to $\Theta$ as an embedding tensor although it is not subject to any representation constraint.
However, it does satisfy the quadratic constraint (\ref{qcon}), and we will se that this
induces 
the structure of a Leibniz algebra on $V$.

\subsection{Algebraic properties of Lie--Leibniz triples}

\begin{defi}\label{def:lieleibniztriple}
A triple $(\fg,V,\Theta)$ consisting of a Lie algebra $\mathfrak{g}$, a $\mathfrak{g}$-module $V$
and a linear map $\Theta:V\to \mathfrak{g}$, 
is a \textbf{Lie--Leibniz triple} if it satisfies
the \textbf{quadratic constraint}
\begin{align}
\Theta(\Theta(u) \cdot v) = [\Theta(u),\Theta(v)]\, \label{quadrconstr0}
\end{align}
for any $u,v\in V$.
The triple $(\fg,V,\Theta)$ is a \textbf{strict Lie--Leibniz triple} if $\Theta$ 
is $\fg$-equivariant in the sense that
\begin{align}
\Theta(x\cdot u)=[x,\Theta(u)] \label{etequivariant}
\end{align}
for any $x\in\fg$ and $u\in V$. A Lie--Leibniz triple $(\fg,V,\Theta)$ is \textbf{surjective} if $\Theta$ is surjective, and \textbf{faithful} if $V$ is faithful as a $\fg$-module (that is, if $\rho$ is injective). The linear map $\Theta:V\to\fg$ which is part
of a Lie--Leibniz triple $(\fg,V,\Theta)$ is called an \textbf{embedding tensor}.
\end{defi}
\noindent
Given a Lie algebra $\mathfrak{g}$ and a $\mathfrak{g}$-module $V$, the existence of 
non-trivial such embedding tensors will not be addressed in the present paper, so we will always assume that at least one exists.

By setting $x=\Theta(u)$ in (\ref{etequivariant}), we immediately see that any 
surjective Lie--Leibniz triple is a strict Lie--Leibniz triple, and any strict Lie--Leibniz triple
is a Lie--Leibniz triple.

\begin{defi}
A \textbf{(left) Leibniz algebra} is an algebra $(V,\bullet)$ such that the Leibniz rule
\begin{align}
u\bullet (v\bullet w)=(u\bullet v)\bullet w + v \bullet (u\bullet w)\, \label{leibniz}
\end{align}
holds for any $u,v,w\in V$.
\end{defi}
\noindent
\begin{remark}\label{twodefsoflieleibniztriple}
In \cite{lavauTensorHierarchiesLeibniz2019,lavauLieAlgebraCrossed2023}, a Lie--Leibniz triple was defined as a triple $(\fg,V,\Theta)$ consisting of a Lie algebra $\mathfrak{g}$, a $\mathfrak{g}$-module $V$
and a linear map $\Theta:V\to \mathfrak{g}$, such that $V$ together with a product $\bullet$ is a Leibniz algebra
satisfying two compability conditions: the \textbf{linear constraint} (which should not be confused with
the representation constraint above)
\begin{align}
u\bullet v = \Theta(u)\cdot v \label{linconstr}
\end{align}
and the \textbf{quadratic constraint}
\begin{align}
\Theta(u \bullet v) = [\Theta(u),\Theta(v)]\,. \label{quadrconstr}
\end{align}
It is obvious that any such triple $(\fg,V,\Theta)$ is a Lie--Leibniz triple according to our Definition~\ref{def:lieleibniztriple}.
Conversely, any Lie--Leibniz triple $(\fg,V,\Theta)$ 
induces the structure of a Leibniz algebra on $V$ 
with the product $\bullet$ defined by $u\bullet v = \Theta(u)\cdot v$, since then
\begin{align}
u \bullet (v \bullet w) &= \Theta(u)\cdot \big(\Theta(v)\cdot w\big)\nn\\
&=[\Theta(u),\Theta(v)]\cdot w + \Theta(v)\cdot \big(\Theta(u)\cdot w\big)\nn\\
&=\Theta(u \bullet v) \cdot w + v \bullet (u\bullet w)\nn\\
&=(u \bullet v) \bullet w + v \bullet (u\bullet w).
\end{align}
Furthermore, the linear constraint (\ref{linconstr}) then
holds by definition, and the quadratic constraint~(\ref{quadrconstr}) is the same as
the quadratic constraint (\ref{quadrconstr0}), which $(\fg,V,\Theta)$ satisfies by Definition~\ref{def:lieleibniztriple}.
Thus the two definitions are indeed 
equivalent.
\end{remark}
\begin{remark}
As stated in Remark~\ref{twodefsoflieleibniztriple}, any Lie--Leibniz triple $(\fg,V,\Theta)$ 
induces the structure of a Leibniz algebra on $V$. 
This observation raises the question whether any Leibniz algebra can be obtained in this way from a
Lie--Leibniz triple $(\fg,V,\Theta)$.
Let $V$ be a Leibniz algebra, and define the (left) adjoint map $\mathrm{ad}: V \to \End{\,V}$
by $\big(\mathrm{ad}{(u)}\big)(v)=u\bullet v$.
We note that the adjoint map factorises through the embedding tensor: $\mathrm{ad}=\rho\circ \Theta$.
It then follows from the Leibniz rule (\ref{leibniz}) that
\begin{align}
\mathrm{ad}\Big(\big(\mathrm{ad}(u)\big)(v)\Big)=[\mathrm{ad}(u),\mathrm{ad}(v)]\,.
\end{align}
Thus the subspace $\mathrm{ad}{(V)}$ of $\End\,{V}$ is a subalgebra of $\mathfrak{gl}(V)$,
and $\big(\mathrm{ad}{(V)},V,\mathrm{ad}\big)$
is a surjective (and thereby strict) Lie--Leibniz triple, which in addition is faithful. 
This Lie--Leibniz triple obviously gives back the
original Leibniz algebra by $\big(\mathrm{ad}(u)\big)(v)=u\bullet v$, so we conclude that indeed any Leibniz algebra
can be constructed from a surjective and faithful Lie--Leibniz triple.
The correspondence is one-to-one up to isomorphism, in the sense that two Lie--Leibniz triples
$(\fg,V,\Theta)$ and $(\fg',V',\Theta')$
are isomorphic if
there is a Lie algebra isomorphism $\varphi: \fg \to \fg'$ 
and a vector space isomorphism $\psi: V \to V'$ such that $\varphi \circ \Theta=\Theta'\circ\psi$.
\end{remark}

\begin{remark}\label{remarkfaith}
Suppose the Lie--Leibniz triple $(\fg,V,\Theta)$ is faithful, and set $V=\scr U_1$.
Then we may identify $\fg$ with a subalgebra of $\scr U_0$ by $x\cdot u = x(u)$
and consider $\Theta$ as an element in $\scr U_{-1}$. Furthermore, the quadratic constraint 
may be written
$[\Theta,\Theta]=0$, since
\begin{align}
\tfrac12[\Theta,\Theta](u)(v)&=[\Theta,\Theta(u)](v)\nn\\
&=[\Theta,\Theta(u)(v)]+[\Theta(v),\Theta(u)]\nn\\
&=\Theta\big(\Theta(u)\cdot v\big)-[\Theta(u),\Theta(v)]\,.
\end{align}
\end{remark}
\noindent
More generally, when $x\cdot u = \rho(x)(u)$ for a representation $\rho$ which is not necessarily faithful,
the quadratic constraint 
may be written
$[\rho\circ\Theta,\rho\circ\Theta]=0$ as 
in (\ref{qcon}).

\subsection{A canonically associated Lie superalgebra} \label{sec:item3}

It has been shown in \cite{lavauLieAlgebraCrossed2023, lavauCorrigendumLieAlgebra2023} that to every
Lie--Leibniz triple  $(\fg,V,\Theta)$  is canonically associated a 
Lie superalgebra $\mathbb{T}$, which in addition to the
structure of Lie superalgebra is equipped with an (inner) differential, so that
it in fact is a differential graded Lie algebra.
Upon factoring out the maximal ideal concentrated in degree $3$ and higher,
we denote the resulting Lie superalgebra ${L}=L(\fg,V,\Theta)$. 
In this section we will show how the latter is connected to Kantor's universal Lie superalgebra associated to the vector space $U_1=V[-1]$, and in particular to $P(V[-1],T_{-1})$ for a certain subspace $T_{-1}$
of $\Hom{(V[-1],\End{\,V[-1]})})$.

We still let $(\fg,V,\Theta)$ be a Lie--Leibniz triple. 
Since $V$ is a $\fg$-module, also the 
vector space 
$\mathrm{Hom}(V,\mathfrak{g})$
is a $\fg$-module with the action of $\fg$
induced by the action of $\mathfrak{g}$ on $V$ and the adjoint action of $\mathfrak{g}$ on itself. More precisely, for any linear map $\phi:V\to\mathfrak{g}$,
any $u\in V$ and any $x\in\fg$ we have 
\begin{equation}\label{eq:defhom}
(x\cdot \phi)(u)=\big[x,\phi(u)\big]-\phi\big(\rho(x)(u)\big)\,.
\end{equation}

We would like to extend the direct sum $\fg \oplus V$ to a 
Lie superalgebra $\mathbb{U}$, 
similar to the universal Lie superalgebra associated to a vector space
$U_1$,
where $\mathbb U_0=\fg$ and $\mathbb U_1\simeq V$. However,
this requires a shift of the degree of $V$.
As graded vector spaces, $\fg$ and $V$ have so far been considered as concentrated in degree $0$ and even.
We now shift the degree of $V$ by $1$ and write
$\mathbb{U}_1=V[-1]$ and  $\mathbb U_0=\fg$. 
We can then, as in Section~\ref{uglsa}, extend $\mathbb U_0 \oplus \mathbb U_1$ 
to a semilocal vector space
$\mathbb U_{1-}$ recursively by
\begin{align}
\mathbb U_{-p+1}=\Hom(\mathbb U_1,\mathbb U_{-p+2})\simeq \Hom(({\mathbb U}_1)^{\otimes p},{\mathbb U}_1)\,
\end{align}
for $p=2,3,\ldots$. In particular, $\mathbb U_{-1}=\Hom{(\mathbb U_1,\fg)}$, which is
$\Hom{(V,\fg)}$ shifted to degree~$-1$.
The embedding tensor $\Theta : V \to \fg$ induces a corresponding shifted map $V[-1] \to \fg$ that we also call embedding tensor
and denote by $\Theta$. Thus we consider $\Theta$ as an element in $\mathbb{U}_{-1}$.

The vector spaces $\mathbb U_{\pm1}$ are $\fg$-modules with the actions of $\fg$ given by the action on
the unshifted vector spaces $V$ and $\Hom{(V,\fg)}$, respectively, that latter in turn given by (\ref{eq:defhom}).
If we set $x(u)=x\cdot u$ for $x\in\fg$ and $u\in \mathbb{U}_1$, we can then prove the following proposition in the same way as Proposition~\ref{prop:semilocal}.
\begin{proposition}
The vector space $\mathbb U_{1-}$ together with the 
morphism
\begin{align}
[\cdot,\cdot]\ :\ (\mathbb U_{1-}\otimes \mathbb U_{1-})_{1-} \to \mathbb U_{1-}
\end{align}
defined recursively by 
\begin{align}
[x, u]  &= x(u)\,, &[u,x]  &=- (-1)^xx(u)\,,
&[x , y](u) &= [x ,y(u)]
+ (-1)^{y} [x(u) , y]
\end{align}
for any $x,y\in \mathbb U_{0-}$ and $u\in \mathbb U_1$
is a semilocal Lie superalgebra.
\end{proposition}
\noindent
In particular, $[x,\phi]=x\cdot \phi$ for any $x\in\mathbb U_0=\fg$ and $\phi\in\mathbb U_{-1}$.
By Proposition \ref{freely}, we can furthermore extend this semilocal Lie superalgebra uniquely to a Lie superalgebra
$\mathbb U$
such that $\mathbb U_+$ is the free Lie superalgebra generated by $\mathbb U_1$.

If the $\fg$-module $V$ is faithful, then this
Lie superalgebra $\mathbb U$ is 
positively $0$-transitive and
isomorphic to a subalgebra of the universal Lie superalgebra $U$ associated to $U_1=\mathbb U_1$ (see Definition \ref{defKantor}).
In fact, it is then isomorphic to the reduced prolongation of $\mathbb U_+$ with respect to $\rho(\fg)$.
However, if $V$ is not faithful, then $\mathbb U$ is positively $(-1)$-transitive, but not
positively $0$-transitive. On the other hand, $U$ is always positively $0$-transitive. 

Let $R_\Theta$ be the cyclic $\mathfrak{g}$-submodule of 
$\mathbb{U}_{-1}=\mathrm{Hom}(\mathbb{U}_1,\mathfrak{g})$
generated by $\Theta$, that is to say: the $\mathfrak{g}$-orbit of the embedding tensor $\Theta$, which is defined as
\begin{equation}
R_\Theta =\mathrm{Span}\big(\left\{ x_1\cdot x_2\cdot \cdots\cdot x_n\cdot \Theta,\ x_1,\ldots,x_n\in\mathfrak{g},\ n\in\mathbb{N}^*\right\}\cup\{\Theta\}\big)\,.
\end{equation}
Let us set $\mathfrak{h}=\Theta(\mathbb{U}_1)$; this is a subalgebra of $\fg$ thanks to the quadratic constraint (\ref{quadrconstr0}). From the same identity, we  observe that the right-hand side of \eqref{eq:defhom} is identically vanishing whenever $x\in \mathfrak{h}$, establishing that the embedding tensor $\Theta$ is $\mathfrak{h}$-equivariant.
However, it is not necessarily $\mathfrak{g}$-equivariant, so $R_\Theta$ is in general more than one-dimensional.

Set $\mathbb{T}_{-1}=R_\Theta$ 
and let $\mathbb{T}_{-}\oplus \mathbb{U}_{0+}$
be the subalgebra of $\mathbb{U}$ generated by $\mathbb{T}_{-1}\oplus \mathbb{U}_0 \oplus \mathbb{U}_1$.
The Lie superalgebra $\mathbb{T}$ is then obtained from $\mathbb{T}_-\oplus \mathbb{U}_{0+}$ by factoring out 
an ideal generated by a particular $\mathfrak{g}$-module at degree $2$.
As we will see, the ideal is included in $\mathbb{U}_{2+}$, so the negative part $\mathbb{T}_-$ is not affected.

The symmetric part of the Leibniz product on $V$ can be understood as a bilinear map 
\begin{align}
\{\cdot,\cdot\}:\quad  S^2V\to V\,,\quad 2\{u,v\}=u\bullet v +v\bullet u\,,
\end{align}
By the identification $\mathbb{U}_1=V[-1]$, 
the bilinear map $\{\cdot,\cdot\}$ induces a degree-$(-1)$ bilinear map
\begin{align}
\overline{\{\cdot,\cdot\}}:\quad \Lambda^2\mathbb{U}_1\to \mathbb{U}_1\,,\quad \overline{\{u,v\}}=\{u[1],v[1]\}[-1]
\end{align}
where $\Lambda^2\mathbb{U}_1$ denotes the graded antisymmetric tensor product (which becomes symmetric since $\mathbb{U}_1$ is odd). Since $\mathbb{U}_+$ is freely generated by $\mathbb{U}_1$,
we can identify 
$\mathbb{U}_2$ with $\Lambda^2\mathbb{U}_1$.
Being induced by the embedding tensor, the bilinear map $\overline{\{\cdot,\cdot\}}$ is $\mathfrak{h}$-equivariant, but not necessarily $\mathfrak{g}$-equivariant.
Then, its kernel
\begin{align}
\mathrm{Ker}\,\overline{\{\cdot,\cdot\}}=\left\{\sum_i\,u_i\wedge v_i \in \Lambda^2\mathbb{U}_1 \,\Big|\, \sum_i\,\overline{\{u_i,v_i\}}=0 \right\}\subseteq \Lambda^2\mathbb{U}_1
\end{align}
is stable under the action of $\mathfrak{h}$ but not necessarily under the action of $\mathfrak{g}$. In
other words, it is 
an $\mathfrak{h}$-submodule of $\Lambda^2\mathbb{U}_1$, but not necessarily a $\mathfrak{g}$-submodule.
Let $K$ be the 
sum of all $\mathfrak{g}$-submodules of $\Lambda^2\mathbb{U}_1$ included in
$\mathrm{Ker}\,\overline{\{\cdot,\cdot\}}$. 
It is a $\mathfrak{g}$-submodule 
of $\mathbb{U}_2=\Lambda^2\mathbb{U}_1$, characterised by the following proposition.
\begin{proposition}
The $\mathfrak{g}$-submodule $K$ is the set
\begin{align} \label{kchar}
K = \{k\in \mathbb{U}_2 \ |\ [\mathbb{T}_{-1},k]=0\}\,.
\end{align}
\end{proposition}
\begin{proof}
It is easy to see that the right hand side of (\ref{kchar}) is a $\fg$-module. Since $\Theta \in R_\Theta$
and $R_\Theta=\mathbb{T}_{-1}$,
it is also a subset of the set
\begin{align}
\{k\in \mathbb{U}_2  \ |\ [\Theta,k]=0\}\,.
\end{align}
But this set is nothing but $\rm{Ker}{\overline{\{\cdot,\cdot\}}}$, since, for every $u,v\in \mathbb{U}_1$ :
\begin{align} \label{Thetauv}
[\Theta,[u,v]]&=[[\Theta,u],v]+[[\Theta,v],u]\nn\\
&=\Theta(u)\cdot v+\Theta(v)\cdot u=2\overline{\{u,v\}}\,.
\end{align}
Thus the right hand side of (\ref{kchar}) is a $\fg$-module included in $\rm{Ker}{\{\cdot,\cdot\}}$,
and since $K$ is the sum of all such $\fg$-modules, we have
\begin{align}
K \supseteq \{k\in \mathbb{U}_2 \ |\ [\mathbb{T}_{-1},k]=0\}\,.
\end{align}
From (\ref{Thetauv}) it also follows that $[\Theta,K]=0$. Now,
let $x\in\mathfrak{g}$ and $k \in K$. Then for every $\chi\in R_\Theta$, we have
\begin{equation}\label{eq:actiong}
[ [x,\chi], k] = [x, [ \chi,k]]-[ \chi, [x,  k]]
\end{equation} 
When $\chi=\Theta$, the first term on the right-hand side vanishes as shown in (\ref{Thetauv})
and the second term vanishes as well because $K$ is a $\mathfrak{g}$-module. This means that 
$[[x,\Theta], k]=0$.
By the same argument, we deduce that $[ [x, [y, \Theta]], k]=0$ for every $x,y\in\mathfrak{g}$. Using this argument inductively, one shows that $[ \chi,k]=0$ for every $\chi\in R_\Theta$ and $k\in K$, that is
\begin{equation}\label{eq:compatibility}
[R_\Theta,K]=0
\end{equation}
and thus also
$K \subseteq \{k\in \mathbb{U}_2 \ | \ [\mathbb{T}_{-1},k]=0\}\,$.
\end{proof}

\noindent
From this, we deduce that the $\mathfrak{g}$-module $K$ generates an ideal $\mathscr{K}$ 
of $\mathbb{T}_-\oplus \mathbb{U}_+$ that is concentrated in degree $2$ and higher.
Factoring out this ideal, we set ${\mathbb{T}}=(\mathbb{T}_-\oplus \mathbb{U}_+)/\mathscr{K}$. 
This is, in reverse grading, the Lie superalgebra $\mathbb{T}$ obtained in \cite{lavauLieAlgebraCrossed2023, lavauCorrigendumLieAlgebra2023}.

The Lie superalgebra $\mathbb{T}$ is positively $(-1)$-transitive, and thus there are no ideals
included in $\mathbb{T}_{-}$. However, there might still be ideals included in $\mathbb{T}_{3+}$.
By factoring out the maximal such ideal we finally obtain the 
Lie superalgebra ${L}=L(\fg,V,\Theta)$.
We immediately get the following proposition.

\begin{proposition}
The Lie superalgebra $L(\fg,V,\Theta)$ is $(-2,2)$-bitransitive.
\end{proposition}

\noindent
We have arrived to the main result of the paper.
\begin{theorem}\label{corolminimal} Under the assumption that $\mathfrak{g}$ is simple,
$V$ is faithful and $\Theta\neq0$,
the Lie superalgebra 
$L=L(\fg,V,\Theta)$ is isomorphic to $P(V[-1],T_{-1})$ for a subspace $T_{-1}$ of $\Hom{(V[-1],\End{\,V[-1]})}$.
\end{theorem}

\begin{proof} Set $U_1=L_1$ and extend $U_1$ to $U$ as in Section~\ref{uglsa}. Since $V$ is faithful, $L$ is positively $0$-transitive. According to Proposition~\ref{314},
the Lie superalgebra $L$ is then embedded in 
the reduced prolongation
of $\scr L_+$ with respect to $\rho(L_{-1})\oplus\rho(L_0)$,
where $\rho:L_{-1}\to U_{1-}$ is the injective semilocal Lie superalgebra morphism defined in Section~\ref{prolsubsec},
extending the representation map associated to the $\fg$-module structure on $V$.
Restricting to the local part,
we recall that it is given by
\begin{align}
\rho(u_1)&=u_1\,,&\rho(x_0)(u_1)&=[x_0,u_1]\,,& \rho(\phi_{-1})(u_1)&=\rho[\phi_{-1},u_1]=\rho(\phi_{-1}(u_1))\,.
\end{align}
If we now set $T_{-1}=\rho(L_{-1})$, we only have to show that $\rho(L_0)$ is the degree-$0$ component of the
subalgebra of $U$ generated by $U_1=\rho(L_{-1})$ and $T_{-1}=\rho(L_1)$ -- in other words,
that $\rho(L_0)=[T_{-1},U_1]$. Since $L_{-1}$ and $L_0$ are $\fg$-modules, $[L_{-1},L_1]$ is an ideal of $\fg$,
which is nontrivial since $\Theta\neq0$. Then, since $\fg$ is simple, $[L_{-1},L_1]=\fg=L_0$ and
\begin{align}
[T_{-1},U_1]=[\rho(L_{-1}),\rho(L_1)]=\rho[L_{-1},L_1]=\rho(L_0)\,.
\end{align}
\end{proof}
 
\subsection{A perspective from differential graded Lie algebras}
 
Let us conclude this section by a few words regarding algebraic properties of the Lie superalgebra $\scr L(\mathfrak{g},V,\Theta)$.
 By the quadratic constraint~\eqref{quadrconstr}, we have that $[ \Theta,\Theta]=0$ (as an element of ${L}_{-2}$). Then, the adjoint action
 \begin{equation}
d_\Theta=[\Theta,\cdot]:L_\bullet\to L_{\bullet-1}
 \end{equation}
of the embedding tensor $\Theta$ on $\scr L(\mathfrak{g},V,\Theta)$
defines an \emph{inner differential} which, by the graded Jacobi identity, is a derivation of  the Lie bracket, turning $\scr L(\mathfrak{g},V,\Theta)$ into
a \textbf{differential graded Lie algebra}.

A broader perspective on this property is based on a general fact concerning differential graded Lie algebras. Whenever one has a differential graded Lie algebra $(L,d,[\cdot,\cdot])$ with differential $d$ (here of degree $-1$), and a \textbf{Maurer--Cartan element}, namely a  degree $-1$ element  $\Theta$ satisfying  the Maurer--Cartan identity
\begin{equation}\label{eqMC}
d\Theta+\tfrac{1}{2}[\Theta,\Theta]=0,
\end{equation}
one can twist the differential $d$ by the adjoint action of this Maurer--Cartan element by setting
$d_\Theta=d+[\Theta,\cdot\,]$. Indeed, for any element $x$ of the given differential graded Lie algebra, one has:
 \begin{align}
 (d_\Theta)^2(x)&=d^2x+d[\Theta,x]+[\Theta,dx]+[\Theta,[\Theta,x]]\nn\\
 &=[d\Theta,x]+\tfrac{1}{2}[[\Theta,\Theta],x]= \left[d\Theta+\tfrac{1}{2}[\Theta,\Theta],x\right]=0\,. 
 \end{align}
 Then, the twisted differential $d_\Theta$ inherits the derivation property (with respect to the Lie bracket) of the original differential $d$, so that $(L,d_\Theta,[\cdot,\cdot])$ is again a differential graded Lie algebra.
 
 In the case at hand, the Lie superalgebra  $\scr L(\mathfrak{g},V,\Theta)$ can originally be seen as a differential graded Lie algebra with zero differential $d=0$.  The Maurer--Cartan identity \eqref{eqMC} is then satisfied by the embedding tensor, since the first term is identically zero, while the second-term is the quadratic constraint. Thus, the inner differential on $\scr L(\mathfrak{g},V,\Theta)$ corresponds to the twisted differential, as we have $d_\Theta=0
 +[\Theta,\cdot]$.

\begin{example}
The choice of ideal $K$ factored out at degree $2$ implies that whenever $V$ is a Lie algebra, then there is no space of degree higher than $1$. A particular case of this situation is the notion of a \textbf{Lie algebra crossed module}, discussed in detail in \cite{lavauLieAlgebraCrossed2023, lavauCorrigendumLieAlgebra2023}.
\end{example}
 
\begin{example}
For an \textbf{augmented Leibniz algebra} \cite{bordemannGlobalIntegrationLeibniz2017}, the embedding tensor is in the singlet representation of the Lie algebra $\fg$ (which we here take to be real for concreteness). In other words, the embedding tensor $\Theta$ is $\fg$-equivariant, so $R_\Theta=\mathbb{R}[1]$. It implies that $K=\rm{Ker}{\overline{\{\cdot,\cdot\}}}$. The quotient $\Lambda^2\mathbb{U}_1/K$ is canonically isomorphic to the subspace $\mathcal{I}\subset V$ generated by all elements of the form $\{u,v\}$. The subspace $\mathcal{I}$ is an ideal of $V$ (with respect to the Leibniz product), that we call the \textbf{ideal of squares} of $V$. Then the associated differential graded Lie algebra $L(\mathfrak{g},V,\Theta)$ does not contain any space in degree higher than $2$ and we have: 
\begin{equation}\label{eq:augmentedLeibnizalg}
\mathcal{I}[-2]\overset{[1]}{\longrightarrow} V[-1]\overset{\Theta}{\longrightarrow}\mathfrak{g}\overset{0}{\longrightarrow} \mathbb{R}[1]
\end{equation}
Every Leibniz algebra canonically gives rise to an augmented Leibniz algebra, by setting $\mathfrak{g}=\mathfrak{inn}(V)$ -- the space of inner derivations of $V$ -- and $\Theta=\mathrm{ad}:u\mapsto \mathrm{ad}\,u$.
\end{example}

\begin{example}
A typical example of an augmented Leibniz algebra can be found in the \textbf{hemi-semidirect product} of a Lie algebra $\mathfrak{g}$ and a $\mathfrak{g}$-module $M$ \cite{kinyonLeibnizAlgebrasCourant2001}. Let us set $V=\mathfrak{g}\oplus M$, which then inherits a $\mathfrak{g}$-module structure as, for any $a\in\mathfrak{g}$, we set
\begin{equation}\label{eqrepresentation}
a\cdot (b,y)=([a,b],a\cdot y)\,.
\end{equation}
Define the embedding tensor $\Theta:V\to \mathfrak{g}$ to be the projection on first factor (the Lie algebra). It satisfies the quadratic constraint and it is $\mathfrak{g}$-equivariant. The Leibniz product on $V$ inherited from the embedding tensor and \eqref{eqrepresentation} is defined as follows:
\begin{equation}
(a,x)\bullet (b,y)=([a,b],a\cdot y)
\end{equation}
This product is exactly the one defining the hemi-semidirect product in \cite{kinyonLeibnizAlgebrasCourant2001} and the data $(\mathfrak{g},V,\Theta)$ is an augmented Leibniz algebra.

Moreover, since the embedding tensor is $\mathfrak{g}$-equivariant, the kernel of the symmetric bracket $\{\cdot,\cdot\}$ is a $\mathfrak{g}$-module, so that $\scr L_2\simeq\mathcal{I}$, the ideal of squares in $V$. On the other hand, the choice of embedding tensor implies that $\mathcal{I}=M$. Then the differential graded Lie algebra \eqref{eq:augmentedLeibnizalg} associated to this hemi-semidirect product is:
\begin{equation}
M[-2]\overset{[1]}{\longrightarrow} (\mathfrak{g}\oplus M)[-1]\overset{\Theta}{\longrightarrow} \mathfrak{g}\overset{0}{\longrightarrow} \mathbb{R}[1]\,.
\end{equation}
\end{example}

 The following table provides in various cases the form of  the  differential graded Lie algebra canonically associated to a given embedding tensor $\Theta\colon V\to \mathfrak{g}$ (where we again take $\fg$ to be real
 for concreteness). More details about the functorial property of this construction can be found in \cite{lavauLieAlgebraCrossed2023, lavauCorrigendumLieAlgebra2023}.
 
\begin{table}[!h]
\centering
\begin{tabular}{l|c|}
\hline
  \multicolumn{1}{|c|}{\begin{tabular}{@{}c@{}}\textbf{embedding}\\\textbf{tensor}\end{tabular}} & \textbf{(differential) graded Lie algebra}\\
  \hline
   \multicolumn{1}{|c|}{\begin{tabular}{@{}c@{}}\textbf{Lie algebra}\\\textbf{crossed}\\\textbf{module} \\ $V\overset{\Theta}{\longrightarrow} \mathfrak{g}$\\
  {\small $V$ is a Lie algebra} \\
  {\small  $\Theta$ is $\mathfrak{g}$-equivariant}
   \end{tabular}}&
   $V[-1]\overset{\Theta}{\rightarrow} \mathfrak{g}\overset{0}{\rightarrow} \mathbb{R}[1] 
   $\\
    \hline
       \multicolumn{1}{|c|}{\begin{tabular}{@{}c@{}}\textbf{augmented}\\\textbf{Leibniz}\\\textbf{algebra}\\
    $V\overset{\Theta}{\rightarrow} \mathfrak{g}$\\
  {\small  $\Theta$ is $\mathfrak{g}$-equivariant}\end{tabular}}&
  $\mathcal{I}[-2]\overset{[1]}{\rightarrow} V[-1]\overset{\Theta}{\rightarrow} \mathfrak{g}\overset{0}{\rightarrow} \mathbb{R}[1]
  $ \\
   \hline
      \multicolumn{1}{|c|}{\begin{tabular}{@{}c@{}}
  $V\overset{\Theta}{\rightarrow} \mathfrak{g}$ \\ {\small  $V$ is a Lie algebra}\end{tabular}}&
   $V[-1]\overset{\Theta}{\rightarrow} \mathfrak{g}\overset{\cdot(-\Theta)}{\rightarrow} 
   R_\Theta\overset{[ \Theta,\cdot\,]}{\rightarrow}[R_\Theta,R_\Theta]\overset{[ \Theta,\cdot\,]}{\rightarrow}\ldots$ \\
  \hline
    \multicolumn{1}{|c|}{\begin{tabular}{@{}c@{}} \textbf{general}\\\textbf{Lie--Leibniz}\\\textbf{triple} \\ $V\overset{\Theta}{\rightarrow} \mathfrak{g}$ 
  \end{tabular} }&
  $\tiny\cdots \rightarrow \Lambda^2\mathbb{U}_1/K\overset{\overline{\{\cdot,\cdot\}}}{\rightarrow}V[-1]\overset{\Theta}{\rightarrow} \mathfrak{g}\overset{\cdot(-\Theta)}{\rightarrow}R_\Theta\overset{[ \Theta,\cdot\,]}{\rightarrow}[R_\Theta,R_\Theta]\overset{[ \Theta,\cdot\,]}{\rightarrow}\cdots$ \\
    \hline
\end{tabular}
\caption{Description of the injective-on-objects function $\mathbf{LieLeib}\rightarrow\mathbf{DGLie}$~\cite{lavauLieAlgebraCrossed2023, lavauCorrigendumLieAlgebra2023}.}
\label{fig1}
\end{table}

\section{Discussion}\label{sec5}

The present paper finally bridges the gap between two graded Lie superalgebras whose respective structures have been inspired by the construction of tensor hierachies in gauged supergravity theories. 
Besides these applications,
the notions of an embedding tensor and a Lie--Leibniz triple have additionally fueled a line of research in abstract algebra, as their properties have been studied further and generalised to more general contexts \cite{kotovEmbeddingTensorLeibniz2020, shengControllingLinfinityAlgebra2021, rongNonabelianembedding2023, tengEmbeddingTensors2024, CaseiroEmbeddingTensors2024, teng3Leibniz2025}.
In particular, the notion of Lie--Leibniz triples, generalising that of augmented Leibniz algebras, paves the way to a better understanding of the \emph{Coquecigrue problem}, as established by Loday \cite{lodayVersionNonComm1993}. This problem consists in finding an appropriate generalisation of Lie's third theorem for Leibniz algebras, that is, in generalising the notion of a Lie group -- a \emph{Coquecigrue}, or \emph{Cocklicrane} in English -- so that Leibniz algebras are their associated tangent structures. After attempts in solving it in special cases  \cite{kinyonLeibnizAlgebrasLie2007}, or only locally \cite{covezLocalIntegrationLeibniz2013}, a global integration to Lie racks has been proposed \cite{bordemannGlobalIntegrationLeibniz2017}. However, it has the drawback of not being functorial, as opposed to Lie's integration from Lie algebras to simply connected Lie groups. A possible way out of that problem is to rely on the correspondence between Lie--Leibniz triples and (differential) graded Lie algebras. Integrating the canonical differential graded Lie algebra \cite{jubinDifferentialGradedLie2022} associated to a given augmented Leibniz algebra may thus offer an alternative answer to the Coquecigrue problem that is functorial. A recent advance in that direction has been to integrate a Lie--Leibniz triple to a Lie-group rack triple \cite{hayami2025integrationlieleibniztriples}. The problem of finding a functorial \emph{Coquecigrue} for Leibniz algebras is left to further investigation.

Another further lead of investigation is the one raised by the observation that the tensor hierarchy gives rise to a particular Lie $\infty$-algebra, whose structure constants are the coefficients of the $p$-form field strength  \cite{greitzTensorHierarchySimplified2014, bonezziLeibnizGaugeTheories2020, lavauInfinityenhancingLeibnizAlgebras2020, bonezziDualityHierarchiesDifferential2021}. It has been shown that this Lie $\infty$-algebra can be implicitly obtained from the graded Lie superalgebra described in Section \ref{sec:item3}  \cite{greitzTensorHierarchySimplified2014, bonezziDualityHierarchiesDifferential2021}, but an explicit mathematical understanding of the construction is not established yet.

\providecommand{\href}[2]{#2}\begingroup\raggedright\endgroup

\end{document}